\documentclass[a4paper, 12pt]{{amsart}}
\usepackage{amsmath, mathtools}
\usepackage{amscd}
\usepackage{amssymb}
\usepackage{amsthm}
\usepackage{ascmac, color}
\usepackage[dvipdfmx]{graphicx, xcolor, pdfpages}
\usepackage[all]{xy}
\usepackage[dvipdfmx]{hyperref}

\newtheorem{lemma}{{\sc Lemma}}[section]
\newtheorem{corollary}[lemma]{{\sc Corollary}}
\newtheorem{proposition}[lemma]{{\sc Proposition}}
\newtheorem{theorem}[lemma]{{\sc Theorem}}
\theoremstyle{definition}
\newtheorem{remark}[lemma]{{\sc Remark}}

\newtheorem{conjecture}[lemma]{{\sc Conjecture}}
\newtheorem{assumption}[lemma]{{\sc Assumption}}

\numberwithin{equation}{section}

\def\Ga{{\mathfrak{a}}}
\def\Gb{{\mathfrak{b}}}

\def\Gg{{\mathfrak{g}}}
\def\Gh{{\mathfrak{h}}}

\def\Gl{{\mathfrak{l}}}
\def\Gm{{\mathfrak{m}}}
\def\Gn{{\mathfrak{n}}}

\def\Gp{{\mathfrak{p}}}

\def\GS{{\mathfrak{S}}}

\def\BA{{\mathbf{A}}}

\def\BC{{\mathbf{C}}}

\def\BF{{\mathbf{F}}}

\def\BQ{{\mathbf{Q}}}
\def\BR{{\mathbf{R}}}

\def\BZ{{\mathbf{Z}}}

\def\Ba{{\mathbf{a}}}

\def\CO{{\mathcal O}}
\def\CI{{\mathcal I}}

\def\CR{{\mathcal R}}
\def\CS{{\mathcal S}}

\def\CX{{\mathcal X}}

\def\ad{{\mathop{\rm ad}\nolimits}}

\def\Comod{\mathop{\rm Comod}\nolimits}
\def\deru{\partial}
\def\diag{{\mathop{\rm diag}\nolimits}}

\def\Hom{\mathop{\rm Hom}\nolimits}

\def\id{\mathop{\rm id}\nolimits}

\def\Ind{{\mathop{\rm Ind}\nolimits}}
\def\inte{{\mathop{\rm int}\nolimits}}
\def\Image{\mathop{\rm Im}\nolimits}
\def\Ker{\mathop{\rm Ker\hskip.5pt}\nolimits}

\def\Mod{\mathop{\rm Mod}\nolimits}

\def\res{{\mathop{\rm res}\nolimits}}

\def\DK{{\rm{DK}}}

\def\txi{\tilde{\xi}}

\begin{document}
\title[an induced module for a quantum group]
{
The Koszul complex and a certain 
induced module for a quantum group}
\author{Toshiyuki TANISAKI}
\begin{abstract}
We give a description of a certain induced module for a quantum group of type $A$.
Together with our previous results this gives 
a proof of Lusztig's conjectural multiplicity formula  for non-restricted modules over the De Concini-Kac type quantized enveloping algebra of type $A_n$ at the $\ell$-th root of unity, where $\ell$ is an odd integer satisfying $(\ell,n+1)=1$ and $\ell> n+1$.
\end{abstract}
\maketitle
\section{Introduction}
\subsection{}
In  our project \cite{T1}, \cite{T2}, \cite{T3}, \cite{T4} 
we pursued  quantum analogues of the results of 
Bezrukavnikov-Mirkovi\'{c}-Rumynin \cite{BMR1}, \cite{BMR2} and  Bezrukavnikov-Mirkovi\'{c} \cite{BM} 
concerning 
$D$-modules on the flag manifolds and representations of the Lie algebras 
in positive characteristics.
More precisely  we  replaced
the flag manifolds in positive characteristics originally employed in \cite{BMR1}, \cite{BMR2}, \cite{BM} 
with the quantized flag manifolds at roots of unity, and tried to deduce results on representations of the De Concini-Kac type quantized enveloping algebras at roots of unity.
So far the only point we are unable to establish  in a sufficient generality  is the description of the cohomology groups of the sheaf of rings of universally twisted differential operators on the quantized flag manifold.
If our conjecture on this problem could be settled, then we would obtain a proof of Lusztig's conjectural multiplicity formula \cite{LK} for non-restricted modules over the De Concini-Kac type quantized enveloping algebra at the $\ell$-th root of unity, where $\ell$ is an odd integer satisfying certain reasonable conditions  (see \cite{T4}).
When $\ell$ is a power of a prime, 
we gave a proof of  Lusztig's formula  in \cite{T4} 
using a weak form of our conjecture proved in \cite{T3}.
In this paper we show the conjecture in the case of type $A_n$ without assuming that $\ell$ is a prime power.
This gives a proof of Lusztig's formula  
in the case of type $A_n$ 
for any odd $\ell$ satisfying $(\ell,n+1)=1$ and $\ell> n+1$.

\subsection{}
\label{subsec:1.2}
Let $G$ be a connected reductive algebraic group over the complex number field $\BC$ with simply connected derived group, and let $B$ be a Borel subgroup of $G$.
For $\zeta\in\BC^\times$ we denote by $O_{\BC,\zeta}(G)$ and 
$O_{\BC,\zeta}(B)$ the quantized coordinate algebras of 
$G$ and $B$ at $q=\zeta$ respectively.
They are Hopf algebras over $\BC$.
Denote by $\Comod(O_{\BC,\zeta}(G))$
(resp.\ $\Comod(O_{\BC,\zeta}(B))$)
the category of right 
$O_{\BC,\zeta}(G)$-comodules
(resp.\ 
$O_{\BC,\zeta}(B)$-comodules).
We have 
the induction functor
\[
\Ind:
\Comod(O_{\BC,\zeta}(B))
\to
\Comod(O_{\BC,\zeta}(G)),
\]
which is left exact.
Then our conjecture on
 the cohomology groups of the sheaf of rings of universally twisted differential operators on the quantized flag manifold follows from the following.
\begin{conjecture}
\label{conj}
Let  $\zeta\in\BC^\times$ be an $\ell$-th root of unity.
We assume that $\ell>1$ is odd and prime to 3 if $G$ contains a component of type $G_2$.
We also assume that $\ell$
is  greater than or equal to the 
Coxeter number of any component of $G$.
Then we have
\[
R^i\Ind(O_{\BC,\zeta}(B)_\ad)
\cong 
\begin{cases}
O_{\BC,\zeta}(G)_\ad\otimes_{O_{\BC,\zeta}(H)^{W\circ}}O_{\BC,\zeta}(H)
\quad&(i=0)
\\
\{0\}&(i>0).
\end{cases}
\]
\end{conjecture}
Here, $O_\ad$ for $O=O_{\BC,\zeta}(G)$, $O_{\BC,\zeta}(B)$ is the right $O$-comodule $O$ with respect to the coadjoint action.
Moreover, $O_{\BC,\zeta}(H)$ denotes the quantized coordinate algebra of a maximal torus $H$ of $G$ and 
$O_{\BC,\zeta}(H)^{W\circ}$ denote the ring of invariants with respect to the twisted  action of the Weyl group $W$ on $O_{\BC,\zeta}(H)$.

Probably the condition that $\ell$ is odd (and prime to 3 for $G_2$) is not necessary in Conjecture \ref{conj}
(see Theorem \ref{thm:intro} below).
I do not know whether the condition on the Coxeter number is really necessary or not.
\subsection{}
In this manuscript we give a proof of Conjecture \ref{conj} in the case $G$ is of type $A$.
In fact our proof works in a more general setting.
The base field  can be arbitrary, and the condition for $\zeta$ is a little more relaxed.
Namely, the main result of this paper is the following.
\begin{theorem}
\label{thm:intro}
Assume that the derived group of $G$ is isomorphic to $SL_n(\BC)$.
Let $K$ be a field and let $\zeta\in K^\times$ be either transcendental over the prime field of $K$ or a root of unity
such that the multiplicative order $\ell$ of $\zeta^2$ is greater than or equal to $n$.
Then we have
\begin{equation}
\label{eq:intro}
R^i\Ind(O_{K,\zeta}(B)_\ad)
\cong 
\begin{cases}
O_{K,\zeta}(G)_\ad\otimes_{O_{K,\zeta}(H)^{W\circ}}O_{K,\zeta}(H)\
\quad&(i=0)
\\
\{0\}&(i>0).
\end{cases}
\end{equation}
\end{theorem}
Here, $O_{K,\zeta}(G)$, etc.\ are obvious analogues of $O_{\BC,\zeta}(G)$, etc.\ over $K$.

Let us describe the outline of the proof.
By induction on the semisimple rank of $G$ we see easily that the proof of  \eqref{eq:intro} is reduced to
showing that 
$R^i\Ind(O_{K,\zeta}(P)_\ad)=0$ for $i\ne0$ 
and that 
the canonical homomorphism
\begin{equation}
\label{eq:intro2}
O_{K,\zeta}(G)_\ad
\otimes_{O_{K,\zeta}(H)^{W\circ}}
O_{K,\zeta}(H)^{W_P\circ}
\to
\Ind(O_{K,\zeta}(P)_\ad)
\end{equation}
is an isomorphism, 
where $P$ is a certain maximal parabolic subgroup of $G$, and $W_P$ is the corresponding subgroup of $W$.
We compute 
$R^i\Ind(O_{K,\zeta}(P)_\ad)$ 
using 
a certain resolution of $O_{K,\zeta}(P)_\ad$ constructed in the following manner.
We can take certain elements $\varphi_1,\dots,\varphi_{n-1}\in O_{K,\zeta}(G)$ 
satisfying
\begin{itemize}
\item[(a)]
$O_{K,\zeta}(G)/
\sum_rO_{K,\zeta}(G)\varphi_r\cong O_{K,\zeta}(P)$,
\item[(b)]
the subalgebra $C$ of $O_{K,\zeta}(G)$ generated by $\varphi_1,\dots,\varphi_{n-1}$ is a quadratic algebra with respect to the generator system $\varphi_1,\dots,\varphi_{n-1}$,
\item[(c)]
the linear span $\bigoplus_r K\varphi_r$ is an
$O_{K,\zeta}(P)$-subcomodule of $O_{K,\zeta}(G)_\ad$.
\end{itemize}
By a general theory of quadratic algebras we obtain a 
Koszul type resolution
\begin{equation}
\label{eq:Koszul}
0\to C{\otimes} L_{n-1}\to\cdots\to C{\otimes}L_{0}\to K\to 0
\end{equation}
of the one-dimensional trivial $C$-module
$
K=C/\sum_rC\varphi_r
$.
Each $L_j$ is naturally an $O_{K,\zeta}(P)$-comodule, and \eqref{eq:Koszul} turns out to be a resolution of the trivial $O_{K,\zeta}(P)$-comodule $K$
with respect to a slightly modified 
coaction of $O_{K,\zeta}(P)$ on the tensor product $C {\otimes} L_j$  (see \cite{HH}, \cite{M2}).
Applying $O_{K,\zeta}(G)\otimes_C(\bullet)$ to \eqref{eq:Koszul} we obtain a resolution 
\begin{equation}
\label{eq:}
0\to O_{K,\zeta}(G)_\ad {\otimes} L_{n-1}\to\cdots\to O_{K,\zeta}(G)_\ad{\otimes} L_{0}\to O_{K,\zeta}(P)_\ad\to 0
\end{equation}
of the $O_{K,\zeta}(P)$-comodule $O_{K,\zeta}(P)_\ad$
with respect to a  modified coaction of $O_{K,\zeta}(P)$.
Since $O_{K,\zeta}(G)_\ad$ is an $O_{K,\zeta}(G)$-comodule, we have
\[
R^i\Ind(O_{K,\zeta}(G)_\ad\otimes L_j)
\cong
O_{K,\zeta}(G)_\ad\otimes R^i\Ind(L_j).
\]
Moreover, by 
the Borel-Weil-Bott type theorem regarding small weights  (see \cite{A}) we have
$R^i\Ind(L_j)=\{0\}$ for $i\ne j$, and 
$R^j\Ind(L_j)=K$.
From this we readily obtain that 
$R^i\Ind(O_{K,\zeta}(P)_\ad)=\{0\}$ for $i\ne0$
and that
$\Ind(O_{K,\zeta}(P)_\ad)$ is endowed with a filtration 
\[
0=M_{-1}\subset M_{0}\subset\cdots
\subset M_{n-1}=\Ind(O_{K,\zeta}(P)_\ad)
\]
such that $M_{r}/M_{r-1}\cong O_{K,\zeta}(G)_{\ad}$ for $r=0,\dots, n-1$.
To show that \eqref{eq:intro2} is an isomorphism we need additional arguments (see Section \ref{sec:main} below for the detail).

\subsection{}
Let $G$ be as in \ref{subsec:1.2} not of type $A$.
Assume that $P$ is a maximal parabolic subgroup of $G$ with commutative unipotent radical $U$.
Let $L$ be the reductive part of $P$.
In this case 
we can take elements $\varphi_1,\dots, \varphi_{\dim U}\in O_{K,\zeta}(G)$ satisfying (a) and (b).
However, 
the linear span $\bigoplus_r K\varphi_r$ is only an
$O_{K,\zeta}(L)$-subcomodule of $O_{K,\zeta}(G)_\ad$ and (c) fails.
This is one of the reasons why our arguments 
are not directly applied to the case $G$ is not of type $A$.

\subsection{}
Let us return to the situation where $G$ is a connected reductive algebraic group over $\BC$ whose derived group is isomorphic to $SL_n(\BC)$.
We assume that $\zeta\in\BC^\times$ is an $\ell$-th root of unity, where $\ell$ is an odd integer satisfying $(\ell,n)=1$ and $\ell> n$.
In this case 
our  Theorem \ref{thm:intro} establishes \cite[Conjecture 6.9]{T3}, 
and this implies \cite[Conjecture 2.6]{T2}
by \cite[Proposition 7.11, Lemma 7.8]{T3}.
Then  by 
\cite[Theorem 2.8]{T2} we obtain a Beilinson-Bernstein type derived equivalence 
between the modules over the De Concini-Kac type quantized enveloping algebras and the $D$-modules on the quantized flag manifold.
By \cite{T4} we obtain Lusztig's conjectural multiplicity formula for the De Concini-Kac type quantized enveloping algebra of type $A$
(see \cite[Section 17.2, 17.3]{LK}, \cite{T4}).

\subsection{}
The contents of this paper are organized as follows.
In Section 2, 3, 4, 5 we recall basic facts on quantum groups and induction functors.
In Section 6 some generality on equivariant modules over the quantized coordinate algebras is presented.
In Section 7  we introduce notation for quantum groups of type $A$.
Our main theorem is  proved in Section 8.

\subsection{}
We use the following notation throughout the paper.
For a ring $R$ we denote by $\Mod(R)$ the category of left $R$-modules.
For a Hopf algebra $U$ we use Sweedler's notation
\[
\Delta^n(u)=
\sum_{(u)}u_{(0)}\otimes\dots\otimes u_{(n)}
\qquad(u\in U)
\]
for the iterated comultiplication $\Delta^n:U\to U^{\otimes n+1}$.

\subsection{}
I would like to thank the referees for valuable comments.

\section{Quantized enveloping algebras}

\subsection{}
Let $\Delta$ be a reduced root system in a vector space $E$ over $\BR$.
For $\alpha\in\Delta$ we denote by $\alpha^\vee\in E^*$ the corresponding coroot.
We choose a set of simple roots 
$\{\alpha_i\mid i\in I\}$, and denote the set of positive roots by $\Delta^+$.
The Weyl group $W$ is the subgroup of $GL(E)$ generated by the simple reflections:
\[
s_i:\lambda\mapsto
\lambda-\langle\lambda,\alpha_i^\vee\rangle\alpha_i
\qquad(i\in I).
\]
We denote by $\ell:W\to\BZ_{\geqq0}$ the length function for the Coxeter system $(W,\{s_i\}_{i\in I})$.
We set
\[
Q=\sum_{\alpha\in\Delta}\BZ\alpha\subset E,
\quad
Q^+=
\sum_{\alpha\in\Delta^+}\BZ_{\geqq0}\alpha\subset E,
\quad
Q^\vee=\sum_{\alpha\in\Delta}\BZ\alpha^\vee
\subset E^*.
\]
For $i, j\in I$ we set 
$a_{ij}=\langle\alpha_j,\alpha_i^\vee\rangle$.

We fix a $W$-invariant symmetric bilinear form 
\begin{equation}
\label{eq:sb}
(\;,\;):E\times E\to \BR
\end{equation}
satisfying 
$(\alpha,\alpha)\in2\BZ_{>0}$ for any $\alpha\in \Delta$ in the following.
Then we have $(Q,Q)\subset\BZ$.
For $i\in I$ we set $d_i=(\alpha_i,\alpha_i)/2$.
\subsection{}
\label{subsec:QE}
For $m\in\BZ_{\geqq0}$ we set 
\[
[m]_t!
=
\prod_{r=1}^m
\frac{t^m-t^{-m}}{t-t^{-1}}
\in\BZ[t,t^{-1}],
\]
where $t$ is an indeterminate.
Let $\BF=\BQ(q)$ be the rational function field in variable $q$ over $\BQ$.
For a free $\BZ$-module $\Gamma$ of finite rank equipped with
$\BZ$-linear maps
\[
Q\xrightarrow{\jmath}\Gamma\xrightarrow{p}
\Hom_\BZ(Q,\BZ)
\]
satisfying
\begin{equation}
\label{eq:Gamma}
\langle(p\circ \jmath)(\alpha),\beta\rangle
=(\alpha,\beta)
\qquad(\alpha,\beta\in Q)
\end{equation}
we define $U_\BF(\Delta,\Gamma)$ to be 
the  $\BF$-algebra
generated by 
\[
k_\gamma\;\;(\gamma\in \Gamma),
\qquad
e_i, f_i\;\;(i\in I)
\]
satisfying the fundamental relations:
\begin{align*}
&k_0=1,\quad
k_{\gamma_1}k_{\gamma_2}=
k_{\gamma_1+\gamma_2}\;\;
&(\gamma_1, \gamma_2\in\Gamma),
\\
&
k_\gamma e_i=
q^{\langle\alpha_i,p(\gamma)\rangle}
e_ik_\gamma,
\quad
k_\gamma f_i=
q^{-\langle\alpha_i,p(\gamma)\rangle}
f_ik_\gamma
&(\gamma\in \Gamma,\; i\in I),
\\
&
e_if_j-f_je_i=\delta_{ij}
\frac{k_i-k_i^{-1}}{q_i-q_i^{-1}}
&(i, j\in I),
\\
&
\sum_{r=0}^{1-a_{ij}}
(-1)^r
e_i^{(1-a_{ij}-r)}e_je_i^{(r)}=0
&(i, j\in I,\; i\ne j),
\\
&
\sum_{r=0}^{1-a_{ij}}
(-1)^r
f_i^{(1-a_{ij}-r)}f_jf_i^{(r)}=0
&(i, j\in I,\; i\ne j).
\end{align*}
Here, 
$q_i=q^{d_i}$, 
$k_i=k_{\jmath(\alpha_i)}$
for $i\in I$, and 
$e_i^{(n)}=e_i^{n}/[n]_{q_i}!$, 
$f_i^{(n)}=f_i^{n}/[n]_{q_i}!$
for $i\in I$, $n\in\BZ_{\geqq0}$.
Then $U_\BF(\Delta,\Gamma)$ turns out to be a Hopf algebra
via
\begin{align*}
&\Delta(k_\gamma)=k_\gamma\otimes k_\gamma,
\quad
\Delta(e_i)=e_i\otimes1+k_i\otimes e_i,
\quad
\Delta(f_i)=f_i\otimes k_i^{-1}+1\otimes f_i,
\\
&\varepsilon(k_y)=1,
\quad
\varepsilon(e_i)=\varepsilon(f_i)=0,
\\
&S(k_\gamma)=k_\gamma^{-1},\quad
S(e_i)=-k_i^{-1}e_i,
\quad
S(f_i)=-f_ik_i
\end{align*}
for $\gamma\in\Gamma$, $i\in I$, 
where $\Delta$, $\varepsilon$, $S$ are the comultiplication, the counit and the antipode respectively.
The Hopf algebra $U_\BF(\Delta,\Gamma)$ is called the quantized enveloping algebra.
We refer the reader to \cite{Jan} for fundamental results on it.

Denote by $U^0_\BF(\Delta,\Gamma)$ the subalgebra of $U_\BF(\Delta,\Gamma)$ generated by $k_\gamma$ for $\gamma\in\Gamma$.
Then we have 
$U_\BF^0(\Delta,\Gamma)=
\bigoplus_{\gamma\in\Gamma}
\BF k_\gamma$.

We have the adjoint action of $U_\BF(\Delta,\Gamma)$ on $U_\BF(\Delta,\Gamma)$ given by
\begin{equation}
\label{eq:ad}
\ad(u)(u')=
\sum_{(u)}u_{(0)}u'(Su_{(1)})
\qquad(u, u'\in U_\BF(\Delta,\Gamma)).
\end{equation}

\subsection{}
Let $G$ be a connected reductive algebraic group over the complex number field $\BC$ with root system $\Delta$.
We fix a maximal torus $H$ of $G$ and denote by 
\[
\CR=(\Lambda,\Delta,\Lambda^\vee,\Delta^\vee)\]
the root datum associated to $G$ and $H$.
Namely, 
$\Lambda$ and $\Lambda^\vee$  are the character group and the cocharacter group of $H$ respectively, 
and $\Delta$, $\Delta^\vee$ are 
the set of  roots and the set of coroots respectively.
We set
\[
\Lambda^+=
\{\lambda\in\Lambda\mid
\langle\lambda,\alpha^\vee\rangle\geqq0
\;\;(\forall\alpha\in\Delta^+)\}
\subset \Lambda.
\]
We denote by $\Gg$ and $\Gh$ the Lie algebras of $G$ and $H$ respectively.
We identify $\Lambda$ and $\Lambda^\vee$ with $\BZ$-lattices of 
$\Gh^*$ and $\Gh$ respectively.
We define subalgebras 
$\Gn$, $\Gb$, $\Gn^+$, $\Gb^+$
of $\Gg$ by
\[
\Gn=\bigoplus_{\alpha\in\Delta^+}\Gg_{-\alpha},
\qquad
\Gb=\Gh\oplus\Gn,
\qquad
\Gn^+=\bigoplus_{\alpha\in\Delta^+}\Gg_{\alpha},
\qquad
\Gb^+=\Gh\oplus\Gn^+,
\]
where 
\[
\Gg_\alpha=
\{x\in\Gg\mid[h,x]=\alpha(h)x\;\;(h\in\Gh)\}
\qquad(\alpha\in\Delta).
\]
We denote by $B$, $B^+$ the closed subgroups of $G$ corresponding to $\Gb$, $\Gb^+$ respectively.

\subsection{}
Note that the condition \eqref{eq:Gamma} is satisfied for $\Gamma=\Lambda^\vee$ 
with
\[
Q\xrightarrow{\jmath_0}\Lambda^\vee\xrightarrow{p_0}
\Hom_\BZ(Q,\BZ),
\]
where $\jmath_0$  is  given by
\[
\jmath_0(\alpha_i)=d_i\alpha_i^\vee\quad(i\in I),
\]
and $p_0$ is the composite of the canonical maps $\Lambda^\vee\cong\Hom_\BZ(\Lambda,\BZ)\to\Hom_\BZ(Q,\BZ)$.
For $\lambda\in\Lambda$ we define an algebra homomorphism 
\begin{equation}
\label{eq:h-char}
\chi_\lambda:U^0_\BF(\Delta,\Lambda^\vee)\to \BF
\end{equation}
by
$\chi_\lambda(k_\gamma)
=q^{\langle\lambda,\gamma\rangle}$ for $\gamma\in\Lambda^\vee$.

Set $\BA=\BZ[q,q^{-1}]$.
We define an $\BA$-subalgebra $U^0_\BA(\CR)$ of 
$U^0_\BF(\Delta,\Lambda^\vee)$ by
\[
U^0_\BA(\CR)=
\{h\in U^0_\BF(\Delta,\Lambda^\vee)\mid
\chi_\lambda(h)\in\BA
\;\;(\forall \lambda\in \Lambda)\}
\]
(see \cite{TA}), 
and define $U_\BA(\CR)$
to be the $\BA$-subalgebra  of 
$U_\BF(\Delta,\Lambda^\vee)$ generated by
$U^0_\BA(\CR)$, $e_i^{(n)}$, $f_i^{(n)}$ for $i\in I$, $n\in\BZ_{\geqq0}$.
It is a Hopf algebra over $\BA$.

For a field $K$ and $\zeta\in K^\times$ we define a Hopf algebra $U_{K,\zeta}(\CR)$ over $K$ by 
\begin{equation}
\label{eq:UR}
U_{K,\zeta}(\CR)=K\otimes_\BA U_\BA(\CR)
\end{equation}
with respect to the ring homomorphism $\BA\to K$ given by $q\mapsto\zeta$.
This Hopf algebra is called the Lusztig type specialization of the quantized enveloping algebra.
Although $U_{K,\zeta}(\CR)$ depends on the root datum $\CR$ (and the choice of  \eqref{eq:sb}), we denote it by $U_{K,\zeta}(\Gg)$ in the rest of this paper.

We denote by $U_{K,\zeta}(\Gh)$ the $K$-subalgebra of $U_{K,\zeta}(\Gg)$ generated by the image of $U^0_\BA(\CR)$.
We have $U_{K,\zeta}(\Gh)
\cong K\otimes_\BA U^0_\BA(\CR)$, and hence 
\eqref{eq:h-char} induces  a $K$-algebra homomorphism
\begin{equation}
\label{eq:h-char2}
\chi_\lambda:U_{K,\zeta}(\Gh)\to K
\qquad(\lambda\in\Lambda).
\end{equation}

For $\Ga=\Gn$, $\Gb$, $\Gn^+$, $\Gb^+$ we define a $K$-subalgebra $U_{K,\zeta}(\Ga)$ 
of 
$U_{K,\zeta}(\Gg)$ by
\begin{align*}
&
U_{K,\zeta}(\Gn)=
\langle f_i^{(n)}\mid i\in I,\; n\in\BZ_{\geqq0}\rangle,
\qquad
U_{K,\zeta}(\Gb)=
\langle
U_{K,\zeta}(\Gh),\;
U_{K,\zeta}(\Gn)
\rangle,
\\
&
U_{K,\zeta}(\Gn^+)=
\langle e_i^{(n)}\mid i\in I,\; n\in\BZ_{\geqq0}\rangle,
\qquad
U_{K,\zeta}(\Gb^+)=
\langle
U_{K,\zeta}(\Gh),\;
U_{K,\zeta}(\Gn^+)
\rangle.
\end{align*}
We also set 
\[
\tilde{U}_{K,\zeta}(\Gn)=S(U_{K,\zeta}(\Gn)).
\]
\begin{remark}
Our conjecture on the cohomology groups of the sheaf of rings of universally twisted differential operators on the quantized flag manifold is naturally formulated in terms of the De Concini-Kac type quantized enveloping algebra $U^{\DK}_{\BC,\zeta}(\Gg)$ (see \ref{subsec:DK} below)
regarded as a $U_{\BC,\zeta}(\Gg)$-module via the adjoint action.
In  reformulating it  into a statement about $O_{\BC,\zeta}(G) \;(\subset U_{\BC,\zeta}(\Gg)^*)$ 
(see Conjecture \ref{conj} and Section \ref{subsec:QC1} below)
we used in \cite{T3} 
the pairing 
\[
\kappa:U_{K,\zeta}(\Gg)\times
U^{\DK}_{K,\zeta}(\Gg)\to K
\]
respecting the adjoint action 
constructed in \cite{TK}.
The pairing $\kappa$ is naturally described using the triangular description \eqref{eq:trig1} below.
This is one of the reasons why we mainly use 
$\tilde{U}_{K,\zeta}(\Gn)$ rather than 
${U}_{K,\zeta}(\Gn)$ in this paper.
\end{remark}

We see easily that $U_{K,\zeta}(\Gh)$, $U_{K,\zeta}(\Gb)$, $U_{K,\zeta}(\Gb^+)$ are Hopf subalgebras of $U_{K,\zeta}(\Gg)$.
The multiplication of $U_{K,\zeta}(\Gg)$ induces 
$K$-module isomorphisms
\begin{equation}
\label{eq:trig1}
{\tilde{U}}_{K,\zeta}(\Gn)
\otimes
U_{K,\zeta}(\Gh)
\otimes
U_{K,\zeta}(\Gn^+)
\cong
U_{K,\zeta}(\Gg),
\end{equation}
\begin{equation}
\label{eq:trig2}
{\tilde{U}}_{K,\zeta}(\Gn)
\otimes
U_{K,\zeta}(\Gh)
\cong
U_{K,\zeta}(\Gb),
\qquad
U_{K,\zeta}(\Gh)
\otimes
U_{K,\zeta}(\Gn^+)
\cong
U_{K,\zeta}(\Gb^+).
\end{equation}

For $i\in I$ denote by $T_i$ the algebra automorphism of $U_\BF(\Delta,\Lambda^\vee)$ given by 
\begin{align*}
T_i(k_\gamma)=&k_{s_i\gamma}
\qquad(\gamma\in\Lambda^\vee),
\\
T_i(e_j)=&
\begin{cases}
-f_ik_i\quad&(j=i)
\\
\sum_{r=0}^{-a_{ij}}(-q_i)^{-r}e_i^{(-a_{ij}-r)}e_je_i^{(r)}\qquad&(j\ne i),
\end{cases}
\\
T_i(f_j)=&
\begin{cases}
-k_i^{-1}e_i\quad&(j=i)
\\
\sum_{r=0}^{-a_{ij}}(-q_i)^{r}f_i^{(r)}f_jf_i^{(-a_{ij}-r)}\qquad&(j\ne i)
\end{cases}
\end{align*}
(see \cite{Lb}).
This gives an action of the braid group associated to $W$ on $U_\BF(\Delta,\Lambda^\vee)$.
Since this action preserves $U_\BA(\CR)$,
we have also an action of the braid group on 
$U_{K,\zeta}(\Gg)$.
In particular, for $w\in W$ we can define an algebra automorphism $T_w$ of $U_{K,\zeta}(\Gg)$
by $T_w=T_{i_1}\cdots T_{i_r}$, where $w=s_{i_1}\dots s_{i_r}$ is a reduced expression of $w$.

The adjoint action \eqref{eq:ad} for $U_\BF(\Delta,\Lambda^\vee)$  induces the adjoint action of 
$U_{K,\zeta}(\Gg)$ on $U_{K,\zeta}(\Gg)$.
We have the weight space decompositions
\[
{\tilde{U}}_{K,\zeta}(\Gn)=
\bigoplus_{\beta\in Q^+}
{\tilde{U}}_{K,\zeta}(\Gn)_{-\beta}, 
\qquad
U_{K,\zeta}(\Gn^+)=
\bigoplus_{\beta\in Q^+}
U_{K,\zeta}(\Gn^+)_{\beta},
\]
where for $U={\tilde{U}}_{K,\zeta}(\Gn)$,  
${U}_{K,\zeta}(\Gn^+)$ we set
\[
U_{\gamma}
=
\{u\in U\mid
\ad(h)(u)=\chi_\gamma(h)u\;\;
(h\in U_{K,\zeta}(\Gh))\}
\qquad(\gamma\in Q).
\]

\subsection{}
For a left (resp.\ right) $U_{K,\zeta}(\Gh)$-module $V$ and $\lambda\in\Lambda$ we set
\begin{gather*}
V_\lambda
=\{v\in V\mid
hv=\chi_\lambda(h)v\;(h\in U_{K,\zeta}(\Gh))\}
\\
(\text{resp.}\;\;
V_\lambda
=\{v\in V\mid
vh=\chi_\lambda(h)v\;(h\in U_{K,\zeta}(\Gh))\}).
\end{gather*}
For $\Ga=\Gg$, $\Gb$, $\Gb^+$,  $\Gh$ we say that a left $U_{K,\zeta}(\Ga)$-module $V$ is diagonalizable if we have
$
V=\bigoplus_{\lambda\in\Lambda}V_\lambda
$
as a $U_{K,\zeta}(\Gh)$-module.
We denote by $\Mod_\diag(U_{K,\zeta}(\Ga))$ the category of diagonalizable left $U_{K,\zeta}(\Ga)$-modules.
We say that a left $U_{K,\zeta}(\Ga)$-module
$V$ is integrable if it is diagonalizable and satisfies
$\dim U_{K,\zeta}(\Ga)v<\infty$ for any $v\in V$.
We denote by $\Mod_\inte(U_{K,\zeta}(\Ga))$ the category of integrable left $U_{K,\zeta}(\Ga)$-modules.
The notion of  integrable right $U_{K,\zeta}(\Ga)$-modules is defined similarly.
We define a left exact functor 
\begin{equation}
\label{eq:finite}
(\bullet)^\inte:\Mod(U_{K,\zeta}(\Ga))
\to
\Mod_\inte(U_{K,\zeta}(\Ga))
\end{equation}
by associating to 
$M\in \Mod(U_{K,\zeta}(\Ga))$
the largest integrable submodule $M^\inte$ of $M$.

For a left $U_{K,\zeta}(\Ga)$-module $V$ its dual space $V^*=\Hom_K(V,K)$ is endowed with a right $U_{K,\zeta}(\Ga)$-module structure via
\[
\langle v^*u,v\rangle
=
\langle v^*, uv\rangle
\qquad
(v^*\in V^*, v\in V, u\in U_{K,\zeta}(\Ga)).
\]
We define a linear map
\[
\Phi_V:V^*\otimes V\to U_{K,\zeta}(\Ga)^*
\]
by 
\[
\langle\Phi_V(v^*\otimes v), u\rangle=
\langle v^*u,v\rangle
=
\langle v^*, uv\rangle
\qquad
(v^*\in V^*, v\in V, u\in U_{K,\zeta}(\Ga)).
\]
The elements of the image of $\Phi_V$ are called the matrix coefficients of the left $U_{K,\zeta}(\Ga)$-module $V$.

For two $U_{K,\zeta}(\Ga)$-modules $V_1$, $V_2$ we regard $V_1\otimes V_2$ as a 
$U_{K,\zeta}(\Ga)$-module using the comultiplication $\Delta:U_{K,\zeta}(\Ga)\to 
U_{K,\zeta}(\Ga)\otimes U_{K,\zeta}(\Ga)$.

We say that a left $U_{K,\zeta}(\Gg)$-module $V$ is a highest weight module with highest weight $\lambda\in\Lambda$ if it is generated by a non-zero element $v$ of $V_\lambda$ satisfying $e_i^{(n)}v=0$ for any $i\in I$ and $n>0$.
Such $v\in V_\lambda$ is called the highest weight vector of $V$.
It is well-known that there exists an integrable highest weight module with highest weight $\lambda$ if and only if $\lambda\in\Lambda^+$.
For $\lambda\in\Lambda^+$ we have a universal integrable highest weight module
\begin{equation}
V_{K,\zeta}(\lambda)
=
U_{K,\zeta}(\Gg)/\CI_\lambda
\end{equation}
with
\[
\CI_\lambda=
\sum_{h\in U_{K,\zeta}(\Gh)}
U_{K,\zeta}(\Gg)(h-\chi_\lambda(h))
+
\sum_{i\in I, n>0}
U_{K,\zeta}(\Gg)e_i^{(n)}
+
\sum_{i\in I, n>\langle\lambda,\alpha_i^\vee\rangle}
U_{K,\zeta}(\Gg)f_i^{(n)},
\]
called the Weyl module.
It has a highest weight vector $v_\lambda=\overline{1}$.
We define a right integrable $U_{K,\zeta}(\Gg)$-module $V_{K,\zeta}^*(\lambda)$ by 
\begin{equation}
V_{K,\zeta}^*(\lambda)=\Hom_K(V_{K,\zeta}(\lambda),K).
\end{equation}
We define $v^*_\lambda\in V_{K,\zeta}^*(\lambda)_\lambda$ by $\langle v^*_\lambda, v_\lambda\rangle=1$.
We occasionally regard 
$V_{K,\zeta}^*(\lambda)$ 
 as a left integrable 
$U_{K,\zeta}(\Gg)$-module by 
\[
\langle uv^*,v\rangle
=
\langle v^*, {}^tuv\rangle
\qquad(v^*\in V_{K,\zeta}^*(\lambda), \; 
u\in U_{K,\zeta}(\Gg),\;
v\in V_{K,\zeta}(\lambda)),
\]
where $u\mapsto {}^tu$ is the algebra anti-automorphism of $U_{K,\zeta}(\Gg)$ given by
\[
h\mapsto h\quad(h\in U_{K,\zeta}(\Gh)),
\qquad
e_i^{(n)}\mapsto f_i^{(n)},\;\;f_i^{(n)}\mapsto e_i^{(n)}
\quad(i\in I, \; n\geqq0).
\]
The submodule $U_{K,\zeta}(\Gg)v^*_\lambda$ of 
$V_{K,\zeta}^*(\lambda)$ is known to be the 
unique (up to isomorphism) irreducible highest weight module with highest weight $\lambda$.

\subsection{}
\label{subsec:DK}
Set $\BA'=\BZ[q,q^{-1},(q_i-q_i^{-1})^{-1}\mid i\in I]\subset\BF$.
For $\Gamma$ as in \ref{subsec:QE} satisfying \eqref{eq:Gamma} 
we define  $U^{\DK}_{\BA'}(\Delta,\Gamma)$
to be the ${\BA'}$-subalgebra  of 
$U_\BF(\Delta,\Gamma)$ generated by
$k_\gamma$, $e_i$, $f_i$ for $\gamma\in\Gamma$, $i\in I$.
It is also a Hopf algebra over ${\BA'}$.
For a field $K$ and $\zeta\in K^\times$ satisfying 
$\zeta^{2d_i}\ne1$ for any $i\in I$
 we define a Hopf algebra $U^\DK_{K,\zeta}(\Delta,\Gamma)$ by 
\[
U^\DK_{K,\zeta}(\Delta,\Gamma)=K\otimes_{\BA'} U^\DK_{\BA'}(\Delta,\Gamma)
\]
with respect to the ring homomorphism ${\BA'}\to K$ given by $q\mapsto\zeta$.
This Hopf algebra is called the De Concini-Kac type specialization of the quantized enveloping algebra.

Note that the condition \eqref{eq:Gamma} is satisfied for $\Gamma=Q$ with
\[
Q\xrightarrow{\id}Q \xrightarrow{p_1}
\Hom_\BZ(Q,\BZ),
\]
where $p_1$ is given by 
\[
\langle p_1(\gamma),\delta\rangle
=
(\gamma,\delta)
\qquad(\gamma, \delta\in Q).
\]
We set $U^\DK_{K,\zeta}(\Gg)=U^\DK_{K,\zeta}(\Delta,Q)$ in the following.

For $\Ga=\Gh$, $\Gn$, $\Gb$ we define a $K$-subalgebra $U^\DK_{K,\zeta}(\Ga)$ 
of 
$U^\DK_{K,\zeta}(\Gg)$ by
\begin{gather*}
U^\DK_{K,\zeta}(\Gh)=
\langle k_\gamma\mid \gamma\in Q\rangle,
\quad
U^\DK_{K,\zeta}(\Gn)=
\langle f_i\mid i\in I\rangle,
\quad
\\
U^\DK_{K,\zeta}(\Gb)=
\langle
U^\DK_{K,\zeta}(\Gh),\;
U^\DK_{K,\zeta}(\Gn)
\rangle.
\end{gather*}
Then $U^\DK_{K,\zeta}(\Gb)$, $U^\DK_{K,\zeta}(\Gh)$ are Hopf subalgebras of $U^\DK_{K,\zeta}(\Gg)$.
The multiplication of $U^\DK_{K,\zeta}(\Gg)$ induces a
$K$-module isomorphism
\begin{gather*}
U^\DK_{K,\zeta}(\Gn)
\otimes
U^\DK_{K,\zeta}(\Gh)
\cong
U^\DK_{K,\zeta}(\Gb).
\end{gather*}
For $\gamma=\sum_{i\in I}m_i\alpha_i\in Q^+$ we denote by $U^\DK_{K,\zeta}(\Gn)_{-\gamma}$ the subspace of $U^\DK_{K,\zeta}(\Gn)$ spanned by the elements of the form 
$f_{i_1}\dots f_{i_N}$ where each $i\in I$ appears $m_i$-times in the sequence $i_1,\dots, i_N$.
We have the weight space decomposition
\[
U^\DK_{K,\zeta}(\Gn)=
\bigoplus_{\gamma\in Q^+}
U^\DK_{K,\zeta}(\Gn)_{-\gamma}.
\]

We have a bilinear form 
\begin{equation}
\label{eq:Dpair}
\tau:U_{K,\zeta}(\Gb^+)\times
U_{K,\zeta}^\DK(\Gb)
\to K,
\end{equation}
called the Drinfeld pairing,
characterized by the properties:
\begin{align*}
&
\tau(x,y_1y_2)=
(\tau\otimes\tau)(\Delta(x),y_2\otimes y_1)
\qquad
&(x\in U_{K,\zeta}(\Gb^+),
\;
y_1, y_2\in U^\DK_{K,\zeta}(\Gb)),
\\
&
\tau(x_1x_2,y)=
(\tau\otimes\tau)(x_1\otimes x_2,\Delta(y))
\qquad
&(x_1, x_2\in U_{K,\zeta}(\Gb^+), 
\;
y\in U^\DK_{K,\zeta}(\Gb)),
\\
&
\tau(h,k_\gamma)=\chi_\gamma(h)
&(\gamma\in Q, h\in U_{K,\zeta}(\Gh)),
\\
&
\tau(h, f_i)=0
&(i\in I, h\in U_{K,\zeta}(\Gh)),
\\
&
\tau(e_i^{(n)}, k_\gamma)=0
&(i\in I, n>0, \gamma\in Q),
\\
&
\tau(e_i^{(n)},f_j)=\delta_{ij}\delta_{1n}\frac1{\zeta^{d_i}-\zeta^{-d_i}}
&(i, j\in I, n\geqq0)
\end{align*}
(see \cite{TK}).
In general this kind of pairing is also called a Hopf skew-pairing.

\section{Quantized coordinate algebras}
\subsection{}\label{subsec:QC1}
Let $\Ga=\Gg$, $\Gb$, $\Gb^+$, $\Gh$, and  
set $A=G$, $B$, $B^+$, $H$ accordingly.
The dual space $U_{K,\zeta}(\Ga)^*=\Hom_K(U_{K,\zeta}(\Ga),K)$ is naturally endowed with a $K$-algebra structure  via the multiplication 
\begin{equation}
\label{eq:O-mult}
\langle\varphi\psi,u\rangle
=
\langle\varphi\otimes\psi,\Delta(u)\rangle
\qquad
(\varphi, \psi\in U_{K,\zeta}(\Ga)^*, u\in U_{K,\zeta}(\Ga)).
\end{equation}
It is also a $U_{K,\zeta}(\Ga)$-bimodule by
\begin{equation}
\label{eq:O-bim}
\langle u_1\varphi u_2,u\rangle
=
\langle \varphi, u_2uu_1\rangle
\qquad
(\varphi\in U_{K,\zeta}(\Ga)^*, 
u_1, u_2, u\in U_{K,\zeta}(\Ga)).
\end{equation}
We denote by $O_{K,\zeta}(A)$ the largest left integrable $U_{K,\zeta}(\Ga)$-submodule 
of 
$U_{K,\zeta}(\Ga)^*$.
It is spanned by the matrix coefficients of finite dimensional left integrable $U_{K,\zeta}(\Ga)$-modules, and
coincides with  the largest right integrable $U_{K,\zeta}(\Ga)$-submodule of 
$U_{K,\zeta}(\Ga)^*$.
We have a natural Hopf algebra structure of  $O_{K,\zeta}(A)$ where 
the multiplication, the unit, the comultiplication,  
the counit, the antipode 
of $O_{K,\zeta}(A)$ are given by 
the transposes of
the comultiplication, the counit, the multiplication,  
the unit, the antipode 
of $U_{K,\zeta}(\Ga)$ respectively.
The following result is standard 
(see for example \cite[Proposition 4.4]{TA}).
\begin{proposition}
\label{prop:comod-int}
The category of right $O_{K,\zeta}(A)$-comodules is naturally equivalent to 
$\Mod_\inte(U_{K,\zeta}(\Ga))$.
\end{proposition}

The inclusions $U_{K,\zeta}(\Gh)\subset 
U_{K,\zeta}(\Gb)
\subset
U_{K,\zeta}(\Gg)$
and 
$U_{K,\zeta}(\Gh)\subset 
U_{K,\zeta}(\Gb^+)
\subset
U_{K,\zeta}(\Gg)$
of Hopf algebras give surjective Hopf algebra homomorphisms
\[
O_{K,\zeta}(G)
\to
O_{K,\zeta}(B)
\to
O_{K,\zeta}(H),
\qquad
O_{K,\zeta}(G)
\to
O_{K,\zeta}(B^+)
\to
O_{K,\zeta}(H).
\]

Let
\[
K[\Lambda]=\bigoplus_{\lambda\in\Lambda}Ke(\lambda)
\]
be the group algebra of $\Lambda$.
Then we have an isomorphism
\[
O_{K,\zeta}(H)
\cong
K[\Lambda]
\qquad
(\chi_\lambda \leftrightarrow e(\lambda))
\]
of $K$-algebras.
We set
\begin{align}
{\tilde{U}}_{K,\zeta}(\Gn)^\bigstar=&
\bigoplus_{\gamma\in Q^+}
({\tilde{U}}_{K,\zeta}(\Gn)_{-\gamma})^*
\subset {\tilde{U}}_{K,\zeta}(\Gn)^*,
\\
{U}_{K,\zeta}(\Gn^+)^\bigstar=&
\bigoplus_{\gamma\in Q^+}
({U}_{K,\zeta}(\Gn^+)_{\gamma})^*
\subset {U}_{K,\zeta}(\Gn^+)^*.
\end{align}
The following fact is proved similarly to \cite[Section 3.3]{TM}, where a closely related result over $\BF$  is shown.

\begin{lemma}
\label{lem:OG}
Regard 
$O_{K,\zeta}(H)$, 
${\tilde{U}}_{K,\zeta}(\Gn)^\bigstar$, 
${U}_{K,\zeta}(\Gn^+)^\bigstar$
as subspaces of $U_{K,\zeta}(\Gg)^*$ via the embeddings
\begin{align}
\label{eq:emb0}
O_{K,\zeta}(H)&\hookrightarrow U_{K,\zeta}(\Gg)^*
\qquad
(\langle\chi, yhx\rangle=\varepsilon(y)\chi(h)\varepsilon(x)),
\\
\label{eq:embm}
{\tilde{U}}_{K,\zeta}(\Gn)^\bigstar&\hookrightarrow U_{K,\zeta}(\Gg)^*
\qquad
(\langle\psi, yhx\rangle=\psi(y)\varepsilon(h)\varepsilon(x)),
\\
\label{eq:embp}
{U}_{K,\zeta}(\Gn^+)^\bigstar&\hookrightarrow U_{K,\zeta}(\Gg)^*
\qquad
(\langle\varphi, yhx\rangle=\varepsilon(y)\varepsilon(h)\varphi(x))
\end{align}
for $y\in {\tilde{U}}_{K,\zeta}(\Gn)$, $h\in {U}_{K,\zeta}(\Gh)$, $x\in {U}_{K,\zeta}(\Gn^+)$
(see \eqref{eq:trig1}).
\begin{itemize}
\item[(i)]
The multiplication of $U_{K,\zeta}(\Gg)^*$ induces the embedding
\begin{equation}
\label{eq:emb}
{\tilde{U}}_{K,\zeta}(\Gn)^\bigstar
\otimes
O_{K,\zeta}(H)
\otimes
U_{K,\zeta}(\Gn^+)^\bigstar
\hookrightarrow
U_{K,\zeta}(\Gg)^*
\end{equation}
of $K$-modules such that 
\[
O_{K,\zeta}(G)\subset
{\tilde{U}}_{K,\zeta}(\Gn)^\bigstar
\otimes
O_{K,\zeta}(H)
\otimes
U_{K,\zeta}(\Gn^+)^\bigstar
\subset
U_{K,\zeta}(\Gg)^*.
\]
\item[(ii)]
\eqref{eq:emb0} is a homomorphism of $K$-algebras.
\item[(iii)]
There exist $K$-algebra structures of 
${\tilde{U}}_{K,\zeta}(\Gn)^\bigstar$ and
${U}_{K,\zeta}(\Gn^+)^\bigstar$  such that
\eqref{eq:embm}, \eqref{eq:embp} are homomorphisms of $K$-algebras.
\item[(iv)]
For $\chi\in O_{K,\zeta}(H)$, 
$\psi\in
{\tilde{U}}_{K,\zeta}(\Gn)^\bigstar$, 
$\varphi\in
({U}_{K,\zeta}(\Gn^+)_{\delta})^*$
$ (\delta\in Q^+)$ we have
\[
\begin{split}
\langle\psi\chi\varphi,yhx\rangle
=&
\langle\psi,y\rangle
\langle k_{\jmath_0(\delta)}\chi,h\rangle
\langle\varphi,x\rangle
\\
&\qquad(
y\in {\tilde{U}}_{K,\zeta}(\Gn),\;
h\in {U}_{K,\zeta}(\Gh),\; 
x\in {U}_{K,\zeta}(\Gn^+))
\end{split}
\]
with respect to the multiplication of ${U}_{K,\zeta}(\Gg)^*$.
\end{itemize}
\end{lemma}
We regard ${\tilde{U}}_{K,\zeta}(\Gn)^\bigstar$, 
$U_{K,\zeta}(\Gn^+)^\bigstar$ as $K$-algebras by
Lemma \ref{lem:OG} (iii).
The multiplication is given by
\begin{align}
\label{eq:mult-n1}
\langle
\psi\psi',y
\rangle
=&\langle
\psi\otimes\psi',\Delta(y)
\rangle
\qquad
(\psi, \psi'\in {\tilde{U}}_{K,\zeta}(\Gn)^\bigstar, \;
y\in {\tilde{U}}_{K,\zeta}(\Gn)),
\\
\label{eq:mult-n2}
\langle
\varphi\varphi',x
\rangle
=&\langle
\varphi\otimes\varphi',\Delta(x)
\rangle
\qquad
(\varphi, \varphi'\in U_{K,\zeta}(\Gn^+)^\bigstar, \;
x\in U_{K,\zeta}(\Gn^+))
\end{align}
using the embeddings \eqref{eq:embm}, 
\eqref{eq:embp}.
We have also the following:
\begin{align}
\label{eq:com1}
\varphi\chi_\lambda
=&
\zeta^{-\langle\lambda,\jmath_0(\gamma)\rangle}
\chi_\lambda\varphi
\qquad(\lambda\in\Lambda, \varphi\in (U_{K,\zeta}(\Gn^+)_\gamma)^*),
\\
\label{eq:com2}
\chi_\lambda\psi
=&
\zeta^{\langle\lambda,\jmath_0(\gamma)\rangle}
\psi\chi_\lambda
\qquad
(\lambda\in\Lambda, 
\psi\in (\tilde{U}_{K,\zeta}(\Gn)_{-\gamma})^*).
\end{align}

By applying the algebra homomorphisms 
$O_{K,\zeta}(G)\to O_{K,\zeta}(B)$ 
and $O_{K,\zeta}(G)\to O_{K,\zeta}(B^+)$ 
to Lemma \ref{lem:OG} we also obtain the following.
\begin{lemma}
\label{lem:OB}
Regard 
$O_{K,\zeta}(H)$, 
${\tilde{U}}_{K,\zeta}(\Gn)^\bigstar$
as subspaces of $U_{K,\zeta}(\Gb)^*$ via the embeddings
\begin{align}
\label{eq:emb0-b}
O_{K,\zeta}(H)&\hookrightarrow U_{K,\zeta}(\Gb)^*
\qquad
(\langle\chi, yh\rangle=\varepsilon(y)\chi(h)),
\\
\label{eq:embm-b}
{\tilde{U}}_{K,\zeta}(\Gn)^\bigstar&\hookrightarrow U_{K,\zeta}(\Gb)^*
\qquad
(\langle\psi, yh\rangle=\psi(y)\varepsilon(h))
\end{align}
for $y\in {\tilde{U}}_{K,\zeta}(\Gn)$, $h\in {U}_{K,\zeta}(\Gh)$ (see \eqref{eq:trig2}).
\begin{itemize}
\item[(i)]
The multiplication of $U_{K,\zeta}(\Gb)^*$ induces the isomorphism
\[
{\tilde{U}}_{K,\zeta}(\Gn)^\bigstar
\otimes
O_{K,\zeta}(H)
\cong
O_{K,\zeta}(B)\subset
U_{K,\zeta}(\Gb)^*
\]
of $K$-modules.
\item[(ii)]
\eqref{eq:emb0-b} and \eqref{eq:embm-b} are homomorphisms of $K$-algebras, where the algebra structure of ${\tilde{U}}_{K,\zeta}(\Gn)^\bigstar$ is given by Lemma \ref{lem:OG} (iii). 
\item[(iii)]
For $\chi\in O_{K,\zeta}(H)$
and
$\psi\in
{\tilde{U}}_{K,\zeta}(\Gn)^\bigstar
$
we have
\[
\langle\psi\chi,yh\rangle
=
\langle\psi,y\rangle
\langle \chi,h\rangle
\quad(
y\in {\tilde{U}}_{K,\zeta}(\Gn),\;
h\in {U}_{K,\zeta}(\Gh))
\]
with respect to the multiplication of ${U}_{K,\zeta}(\Gb)^*$.
\end{itemize}
\end{lemma}
\begin{lemma}
\label{lem:OB+}
Regard 
$O_{K,\zeta}(H)$, 
${{U}}_{K,\zeta}(\Gn^+)^\bigstar$
as subspaces of $U_{K,\zeta}(\Gb^+)^*$ via the embeddings
\begin{align}
\label{eq:emb0-b+}
O_{K,\zeta}(H)&\hookrightarrow U_{K,\zeta}(\Gb^+)^*
\qquad
(\langle\chi, hx\rangle=\chi(h)\varepsilon(x)),
\\
\label{eq:embm-b+}
{{U}}_{K,\zeta}(\Gn^+)^\bigstar&\hookrightarrow U_{K,\zeta}(\Gb^+)^*
\qquad
(\langle\varphi, hx\rangle=\varepsilon(h)\varphi(x))
\end{align}
for $x\in {{U}}_{K,\zeta}(\Gn^+)$, $h\in {U}_{K,\zeta}(\Gh)$ (see \eqref{eq:trig2}).
\begin{itemize}
\item[(i)]
The multiplication of $U_{K,\zeta}(\Gb^+)^*$ induces the isomorphism
\[
O_{K,\zeta}(H)\otimes
{{U}}_{K,\zeta}(\Gn^+)^\bigstar
\cong
O_{K,\zeta}(B^+)\subset
U_{K,\zeta}(\Gb^+)^*
\]
of $K$-modules.
\item[(ii)]
\eqref{eq:emb0-b+} and \eqref{eq:embm-b+} are homomorphisms of $K$-algebras, where the algebra structure of ${{U}}_{K,\zeta}(\Gn^+)^\bigstar$ is given by Lemma \ref{lem:OG} (iii). 
\item[(iii)]
For $\chi\in O_{K,\zeta}(H)$
and
$\varphi\in
({{U}}_{K,\zeta}(\Gn^+)_{\delta})^*
$ ($\delta\in Q^+$)
we have
\[
\langle\chi\varphi,hx\rangle
=
\langle k_{\jmath_0(\delta)}\chi,h\rangle
\langle\varphi,x\rangle
\quad(
h\in {U}_{K,\zeta}(\Gh),\;
x\in {{U}}_{K,\zeta}(\Gn^+))
\]
with respect to the multiplication of ${U}_{K,\zeta}(\Gb^+)^*$.
\end{itemize}
\end{lemma}
\begin{lemma}
\label{lem:duality}
Assume 
$\zeta^{2d_i}\ne1$ for any $i\in I$.
Then there exists an anti-isomorphism
\[
F:U^\DK_{K,\zeta}(\Gn)\to U_{K,\zeta}(\Gn^+)^\bigstar
\]
of $K$-algebras satisfying 
$F(U^\DK_{K,\zeta}(\Gn)_{-\gamma})=(U_{K,\zeta}(\Gn^+)_\gamma)^*$ for any $\gamma\in Q^+$.
\end{lemma}
\begin{proof}
The desired anti-isomorphism $F$ is given by
\[
\langle F(y),x\rangle=
\tau(x,y)
\qquad(y\in U^\DK_{K,\zeta}(\Gn), 
x\in U_{K,\zeta}(\Gn^+)).
\]
In fact we see easily from the characterization of the Drinfeld pairing that $F$ is an anti-homomorphism.
We can also show that $F$ is bijective  using the well-known formula for the values of the Drinfeld pairing on the PBW-bases
(see  \cite{KR}, \cite{LS}, \cite{Lb}, \cite{TD}).
\end{proof}

We see easily that for $\lambda\in\Lambda^+$ we have
\[
\chi_\lambda=\Phi_{V_{K,\zeta}(\lambda)}(v^*_\lambda\otimes v_\lambda)
\in O_{K,\zeta}(G)\subset U_{K,\zeta}(\Gg)^*,
\]
where $O_{K,\zeta}(H)$ is regarded as a subalgebra of 
$U_{K,\zeta}(\Gg)^*$
via the embedding \eqref{eq:emb0}.
We set 
\begin{equation}
\CS=\{\chi_\lambda\mid\lambda\in \Lambda^+\}
\subset
{O}_{K,\zeta}(G).
\end{equation}
It is known that the multiplicative subset 
$\CS$ satisfies the left and right Ore conditions in $O_{K,\zeta}(G)$ (see \cite{J}, \cite[Corollary 3.12]{TM}).
We define a $K$-algebra $\hat{O}_{K,\zeta}(G)$ by
\begin{equation}
\hat{O}_{K,\zeta}(G)=\CS^{-1}{O}_{K,\zeta}(G).
\end{equation}
It is a subalgebra of $U_{K,\zeta}(\Gg)^*$.
It is easily seen that 
the image of 
\eqref{eq:emb}
coincides with $\hat{O}_{K,\zeta}(G)$ (see \cite[Proposition 3.13]{TM}).
Hence the multiplication of $\hat{O}_{K,\zeta}(G)$ induces the  
isomorphism
\begin{equation}
\label{eq:C-tri}
{\tilde{U}}_{K,\zeta}(\Gn)^\bigstar
\otimes
O_{K,\zeta}(H)
\otimes
U_{K,\zeta}(\Gn^+)^\bigstar 
\cong\hat{O}_{K,\zeta}(G)
\end{equation}
of $K$-modules.
\subsection{}
We have
\begin{equation}
\label{eq:F1}
u(\varphi\varphi')=
\sum_{(u)}
(u_{(0)}\varphi)(u_{(1)}\varphi'),
\qquad
(\varphi\varphi')u=
\sum_{(u)}
(\varphi u_{(0)})(\varphi'u_{(1)})
\end{equation}
for 
$u\in U_{K,\zeta}(\Gg)$, 
$\varphi, \varphi'\in U_{K,\zeta}(\Gg)^*$
with respect to the 
$U_{K,\zeta}(\Gg)$-bimodule structure \eqref{eq:O-bim} of $U_{K,\zeta}(\Gg)^*$.
Note that  $O_{K,\zeta}(G)$ is a $U_{K,\zeta}(\Gg)$-subbimodule
of $U_{K,\zeta}(\Gg)^*$.
It is also easily seen that  
$\hat{O}_{K,\zeta}(G)$ is a $U_{K,\zeta}(\Gg)$-subbimodule
of $U_{K,\zeta}(\Gg)^*$.

We define the adjoint action of 
$U_{K,\zeta}(\Gg)$ on $U_{K,\zeta}(\Gg)^*$ by
\begin{equation}
\label{eq:adOG}
\ad(u)(\varphi)=\sum_{(u)}u_{(0)}\varphi(S^{-1}u_{(1)})
\qquad
(u\in U_{K,\zeta}(\Gg),\; \varphi\in U_{K,\zeta}(\Gg)^*).
\end{equation}
This gives the adjoint action of 
$U_{K,\zeta}(\Gg)$ on $O_{K,\zeta}(G)$ and $\hat{O}_{K,\zeta}(G)$.
By \eqref{eq:F1} we have 
\begin{equation}
\label{eq:F2}
\ad(u)(\varphi\varphi')
=
\sum_{(u)}
(u_{(0)}\varphi(S^{-1}u_{(2)}))
(\ad(u_{(1)})(\varphi'))
\end{equation}
for $\varphi, \varphi'\in U_{K,\zeta}(\Gg)^*$, $u\in U_{K,\zeta}(\Gg)$.

\section{Parabolic subalgebras}
\subsection{}
For $J\subset I$
we set
\[
\Delta_J=\sum_{j\in J}\BZ\alpha_j\cap\Delta,
\qquad
\Delta_J^+=\Delta_J\cap \Delta^+,
\qquad
W_J=\langle s_j\mid j\in J\rangle\subset W,
\]
and define subalgebras
$\Gp_J$, $\Gl_J$, $\Gb_J$, $\Gn_J$, $\Gn^+_J$, $\Gm_J$ of $\Gg$ by 
\begin{gather*}
\Gp_J=\Gb\oplus
\left(
\bigoplus_{\alpha\in\Delta_J^+}\Gg_\alpha
\right),
\qquad
\Gl_J=\Gh\oplus
\left(
\bigoplus_{\alpha\in\Delta_J}\Gg_\alpha
\right),
\qquad
\Gb_J=
\Gh\oplus
\left(
\bigoplus_{\alpha\in\Delta_J^+}\Gg_{-\alpha}
\right),
\\
\Gn_J=
\bigoplus_{\alpha\in\Delta_J^+}
\Gg_{-\alpha},
\qquad
\Gn^+_J=
\bigoplus_{\alpha\in\Delta_J^+}
\Gg_{\alpha},
\qquad
\Gm_J=
\bigoplus_{\alpha\in\Delta^+\setminus\Delta_J^+}
\Gg_{-\alpha}.
\end{gather*}
respectively.
Note that $\Gp_\emptyset=\Gb$ and $\Gp_I=\Gg$.
We denote by $P_J$, $L_J$, $B_J$ the closed connected subgroups of $G$ with Lie algebras 
$\Gp_J$, $\Gl_J$, $\Gb_J$ respectively.

\subsection{}
Let $\Ga=\Gp_J$, $\Gl_J$, $\Gb_J$, 
and write $A=P_J$, $L_J$, $B_J$ accordingly.
We define a subalgebra $U_{K,\zeta}(\Ga)$ of $U_{K,\zeta}(\Gg)$ by
\begin{align*}
U_{K,\zeta}(\Gp_J)=&
\langle U_{K,\zeta}(\Gh), \;
f_i^{(n)}, \; e_j^{(n)}
\mid
i\in I,\; j\in J,\; n\geqq0
\rangle,
\\
U_{K,\zeta}(\Gl_J)=&
\langle U_{K,\zeta}(\Gh), \;
f_j^{(n)}, \; e_j^{(n)}
\mid
j\in J,\; n\geqq0
\rangle,
\\
U_{K,\zeta}(\Gb_J)=&
\langle U_{K,\zeta}(\Gh), \;
f_j^{(n)}
\mid
j\in J,\; n\geqq0
\rangle.
\end{align*}
It is a Hopf algebra.
We have 
\begin{equation}
\label{eq:CRJ}
U_{K,\zeta}(\Gl_J)=U_{K,\zeta}(\CR_J)
\end{equation}
for the root datum
$
\CR_J=(\Lambda,\Delta_J,\Lambda^\vee, \Delta_J^\vee)$ using the notation \eqref{eq:UR}.
We define the notion of integrable 
(resp.\ diagonalizable) left and right
$U_{K,\zeta}(\Ga)$-modules similarly to the case $
\Ga=\Gg$, $\Gb$, and denote by 
$\Mod_\inte(U_{K,\zeta}(\Ga))$ 
(resp.\ $\Mod_\diag(U_{K,\zeta}(\Ga))$)
the category of integrable 
(resp.\ diagonalizable) left
$U_{K,\zeta}(\Ga)$-modules.
We define $O_{K,\zeta}(A)$ to be the largest left integrable $U_{K,\zeta}(\Ga)$-submodule of $U_{K,\zeta}(\Ga)^*$.
It is also the  largest right integrable $U_{K,\zeta}(\Ga)$-submodule of $U_{K,\zeta}(\Ga)^*$.
We have a natural Hopf algebra structure of $O_{K,\zeta}(A)$.

We denote by 
\begin{equation}
\pi_J:U_{K,\zeta}(\Gp_J)
\to
U_{K,\zeta}(\Gl_J)
\end{equation}
the algebra homomorphism given by 
\[
\pi_J(a)=a\quad(a\in U_{K,\zeta}(\Gl_J)),
\qquad
\pi_J(f_i^{(n)})=0
\quad(i\in I\setminus J,\; n>0).
\]
It is a Hopf algebra homomorphism and induces the embedding
\begin{equation}
\pi_J^*:O_{K,\zeta}(L_J)
\hookrightarrow
O_{K,\zeta}(P_J)
\end{equation}
of Hopf algebras.

We set 
\begin{align*}
\tilde{U}_{K,\zeta}(\Gn_J)
=
\langle Sf_j^{(n)}\mid j\in J,\; n\geqq0\rangle,
\qquad
{U}_{K,\zeta}(\Gn^+_J)
=
\langle e_j^{(n)}\mid j\in J,\; n\geqq0\rangle.
\end{align*}
They are subalgebras of 
$\tilde{U}_{K,\zeta}(\Gn)$ and 
${U}_{K,\zeta}(\Gn^+)$ respectively.
We have isomorphisms
\begin{align}
\label{eq:tri-p1}
U_{K,\zeta}(\Gp_J)
\cong&
\tilde{U}_{K,\zeta}(\Gn)\otimes
U_{K,\zeta}(\Gh)\otimes
{U}_{K,\zeta}(\Gn^+_J),
\\
\label{eq:tri-p2}
U_{K,\zeta}(\Gl_J)
\cong&
\tilde{U}_{K,\zeta}(\Gn_J)\otimes
U_{K,\zeta}(\Gh)\otimes
{U}_{K,\zeta}(\Gn^+_J)
\end{align}
given by the multiplication of $U_{K,\zeta}(\Gg)$.
We also define a subalgebra
$\tilde{U}_{K,\zeta}(\Gm_J)$ of 
$\tilde{U}_{K,\zeta}(\Gn)$ by
\begin{equation}
\tilde{U}_{K,\zeta}(\Gm_J)
=
\tilde{U}_{K,\zeta}(\Gn)
\cap
T_{w_J}^{-1}(\tilde{U}_{K,\zeta}(\Gn)),
\end{equation}
where $w_J$ is the longest element of 
$W_J$.
The multiplication of $U_{K,\zeta}(\Gg)$ induces
\[
\tilde{U}_{K,\zeta}(\Gm_J)
\otimes
\tilde{U}_{K,\zeta}(\Gn_J)
\cong
\tilde{U}_{K,\zeta}(\Gn)
\]
(see \cite[Proposition 2.10]{TM}), 
and hence we have
\begin{equation}
\label{eq:trig-p}
\tilde{U}_{K,\zeta}(\Gm_J)
\otimes
{U}_{K,\zeta}(\Gl_J)
\cong
{U}_{K,\zeta}(\Gp_J),
\qquad
\tilde{U}_{K,\zeta}(\Gm_J)
\otimes
{U}_{K,\zeta}(\Gb_J)
\cong
{U}_{K,\zeta}(\Gb).
\end{equation}
By \cite[Lemma 2.8]{TM} we have
\begin{equation}
\label{eq:mJ-comult}
\Delta({\tilde{U}}_{K,\zeta}(\Gm_J))
\subset
U_{K,\zeta}(\Gb)\otimes 
{\tilde{U}}_{K,\zeta}(\Gm_J).
\end{equation}
We have the weight space decompositions
\[
{U}_{K,\zeta}(\Gn^+_J)
=
\bigoplus_{\gamma\in Q^+}
{U}_{K,\zeta}(\Gn^+_J)_{\gamma},
\qquad
\tilde{U}_{K,\zeta}(\Gm_J)
=
\bigoplus_{\gamma\in Q^+}
\tilde{U}_{K,\zeta}(\Gm_J)_{-\gamma}
\]
with respect to the adjoint action of ${U}_{K,\zeta}(\Gh)$.

\subsection{}
We set
\begin{align*}
{U}_{K,\zeta}(\Gn^+_J)^\bigstar
=&
\bigoplus_{\gamma\in Q^+}
({U}_{K,\zeta}(\Gn^+_J)_{\gamma})^*
\subset
{U}_{K,\zeta}(\Gn^+_J)^*,
\\
\tilde{U}_{K,\zeta}(\Gm_J)^\bigstar
=&
\bigoplus_{\gamma\in Q^+}
(\tilde{U}_{K,\zeta}(\Gm_J)_{-\gamma})^*
\subset
\tilde{U}_{K,\zeta}(\Gm_J)^*.
\end{align*}

\begin{lemma}
\label{lem:OP}
Regard 
$O_{K,\zeta}(L_J)$, 
${\tilde{U}}_{K,\zeta}(\Gm_J)^\bigstar$
as subspaces of $U_{K,\zeta}(\Gp_J)^*$ via the embeddings
\begin{align}
\label{eq:embp0}
O_{K,\zeta}(L_J)&\hookrightarrow U_{K,\zeta}(\Gp_J)^*
\qquad
(\langle\varphi, ya\rangle=\varepsilon(y)\varphi(a)),
\\
\label{eq:embpm}
{\tilde{U}}_{K,\zeta}(\Gm_J)^\bigstar&\hookrightarrow U_{K,\zeta}(\Gp_J)^*
\qquad
(\langle\psi, ya\rangle=\psi(y)\varepsilon(a))
\end{align}
for $y\in {\tilde{U}}_{K,\zeta}(\Gm_J)$, $a\in {U}_{K,\zeta}(\Gl_J)$
(see \eqref{eq:trig-p}).
\begin{itemize}
\item[(i)]
The multiplication of $U_{K,\zeta}(\Gp_J)^*$ induces the isomorphism
\begin{equation}
\label{eq:embPP}
{\tilde{U}}_{K,\zeta}(\Gm_J)^\bigstar
\otimes
O_{K,\zeta}(L_J)
\cong
O_{K,\zeta}(P_J)
\subset
U_{K,\zeta}(\Gp_J)^*
\end{equation}
of $K$-modules.
\item[(ii)]
\eqref{eq:embp0} is a homomorphism of $K$-algebras.
\item[(iii)]
There exists a natural $K$-algebra structure of 
${\tilde{U}}_{K,\zeta}(\Gm_J)^\bigstar$  such that
\eqref{eq:embpm} is a homomorphism of $K$-algebras.
\item[(iv)]
For $\varphi\in O_{K,\zeta}(L_J)$, 
$\psi\in
{\tilde{U}}_{K,\zeta}(\Gm_J)^\bigstar$ we have
\[
\langle\psi\varphi,ya\rangle
=
\langle\psi,y\rangle
\langle\varphi,a\rangle
\qquad(
y\in {\tilde{U}}_{K,\zeta}(\Gm_J),\;
a\in {U}_{K,\zeta}(\Gl_J))
\]
with respect to the multiplication of ${U}_{K,\zeta}(\Gp_J)^*$.
\end{itemize}
\end{lemma}
\begin{proof}
(iv) By \eqref{eq:mJ-comult}  we have
\begin{align*}
\langle\psi\varphi,ya\rangle
=&\sum_{(y),(a)}
\langle\psi,y_{(0)}a_{(0)}\rangle
\langle\varphi,y_{(1)}a_{(1)}\rangle
=\sum_{(y),(a)}
\varepsilon(a_{(0)})
\varepsilon(y_{(1)})
\langle\psi,y_{(0)}\rangle
\langle\varphi,a_{(1)}\rangle
\\
=&\langle\psi,y\rangle
\langle\varphi,a\rangle.
\end{align*}

(ii) This follows from the fact that \eqref{eq:embp0} is the composite of $\pi_J^*:O_{K,\zeta}(L_J)\to O_{K,\zeta}(P_J)$ and $O_{K,\zeta}(P_J)\hookrightarrow {U}_{K,\zeta}(\Gp_J)^*$.

(iii) We have obviously $1_{O_{K,\zeta}(P_J)}=\varepsilon\in \tilde{U}_{K,\zeta}(\Gm_J)^\bigstar$.
Let $\psi, \psi'\in \tilde{U}_{K,\zeta}(\Gm_J)^\bigstar\subset {U}_{K,\zeta}(\Gp_J)^*$, and denote by $\psi\psi'$ their product inside ${U}_{K,\zeta}(\Gp_J)^*$.
Then for $y\in \tilde{U}_{K,\zeta}(\Gm_J)$, 
$a\in {U}_{K,\zeta}(\Gl_J)$
we have
\begin{align*}
\langle\psi\psi',ya\rangle
=&\sum_{(y),(a)}
\langle\psi,y_{(0)}a_{(0)}\rangle
\langle\psi',y_{(1)}a_{(1)}\rangle
=\sum_{(y),(a)}
\varepsilon(a_{(0)})
\varepsilon(a_{(1)})
\langle\psi,y_{(0)}\rangle
\langle\psi',y_{(1)}\rangle
\\
=&
\varepsilon(a)
\langle\psi\psi',y\rangle,
\end{align*}
and hence we obtain $\psi\psi'\in \tilde{U}_{K,\zeta}(\Gm_J)^\bigstar\subset {U}_{K,\zeta}(\Gp_J)^*$.

(i) By (iv) we have a natural embedding 
${\tilde{U}}_{K,\zeta}(\Gm_J)^\bigstar
\otimes
O_{K,\zeta}(L_J)
\hookrightarrow
U_{K,\zeta}(\Gp_J)^*$.
Moreover, it is easily seen that 
\[
O_{K,\zeta}(P_J)\subset
{\tilde{U}}_{K,\zeta}(\Gm_J)^\bigstar
\otimes
O_{K,\zeta}(L_J)
\subset
U_{K,\zeta}(\Gp_J)^*.
\]
Since $O_{K,\zeta}(P_J)$ is a subalgebra of 
$U_{K,\zeta}(\Gp_J)^*$, it is sufficient to show that ${\tilde{U}}_{K,\zeta}(\Gm_J)^\bigstar$ and 
$O_{K,\zeta}(L_J)$ are contained in 
$O_{K,\zeta}(P_J)$.
By the proof of (ii) we have $O_{K,\zeta}(L_J)\subset O_{K,\zeta}(P_J)$.
It remains to show 
${\tilde{U}}_{K,\zeta}(\Gm_J)^\bigstar
\subset O_{K,\zeta}(P_J)$.

Set 
\[
\Lambda^J
=\{\lambda\in\Lambda\mid
\langle\lambda,\alpha_j^\vee\rangle=0\;(j\in J)\}.
\]
For $\lambda\in\Lambda^J$ we have a character $\chi^J_\lambda:U_{K,\zeta}(\Gl_J)\to K$ given by 
\[
h\mapsto \chi_\lambda(h)\quad(h\in U_{K,\zeta}(\Gh)),\qquad
e_j^{(n)}, f_j^{(n)}\mapsto 0\quad(j\in J, n\geqq0).
\]
We have obviously $\chi^J_\lambda\in O_{K,\zeta}(L_J)$ for $\lambda\in \Lambda^J$.
For $\lambda\in\Lambda^J\cap\Lambda^+$ we have 
\[
V_{K,\zeta}(\lambda)=
\tilde{U}_{K,\zeta}(\Gn)v_\lambda
=
\tilde{U}_{K,\zeta}(\Gm_J)
\tilde{U}_{K,\zeta}(\Gn_J)v_\lambda
=
\tilde{U}_{K,\zeta}(\Gm_J)v_\lambda.
\]
Now let us show $({\tilde{U}}_{K,\zeta}(\Gm_J)_{-\gamma})^*
\subset O_{K,\zeta}(P_J)$ for any $\gamma\in Q^+$.
Let $\gamma\in Q^+$.
Similarly to \cite[Proposition 1.4.1]{TK}
we can  show  that
there exists some $\lambda\in\Lambda^J\cap \Lambda^+$ such that
\[
{\tilde{U}}_{K,\zeta}(\Gm_J)_{-\gamma}\to 
V_{K,\zeta}(\lambda)_{\lambda-\gamma}
\qquad(u\mapsto uv_\lambda)
\]
is bijective
(see also the proof of \cite[Proposition 1.10]{L0}).
Hence for any $\psi\in({\tilde{U}}_{K,\zeta}(\Gm_J)_{-\gamma})^*$
there exists a matrix coefficient $\varphi$ of $V_{K,\zeta}(\lambda)$ satisfying
\[
\langle\varphi,ya\rangle=
\langle\psi,y\rangle
\langle\chi^J_\lambda,a\rangle
\qquad(y\in\tilde{U}_{K,\zeta}(\Gm_J), \;
a\in {U}_{K,\zeta}(\Gl_J)).
\]
The  restriction $\overline{\varphi}$ of $\varphi$ to ${U}_{K,\zeta}(\Gp_J)$ belongs to $O_{K,\zeta}(P_J)$ and we have
$\overline{\varphi}=\psi\chi^J_\lambda$.
Hence we obtain 
$\psi=\overline{\varphi}\chi^J_{-\lambda}\in O_{K,\zeta}(P_J)$.
\end{proof}
We regard ${\tilde{U}}_{K,\zeta}(\Gm_J)^\bigstar$ 
as a $K$-algebra by
Lemma \ref{lem:OP} (iii).
The multiplication is given by
\[
\langle
\psi\psi',y
\rangle
=\langle
\psi\otimes\psi',\Delta(y)
\rangle
\qquad
(\psi, \psi'\in {\tilde{U}}_{K,\zeta}(\Gm_J)^\bigstar, \;
y\in {\tilde{U}}_{K,\zeta}(\Gm_J))
\]
using the embedding \eqref{eq:embpm}.
We denote by 
$\hat{O}_{K,\zeta}(P_J)$ the image of the composite of the inclusion
$\hat{O}_{K,\zeta}(G)\hookrightarrow
U_{K,\zeta}(\Gg)^*$ and 
the restriction map
$U_{K,\zeta}(\Gg)^*\to U_{K,\zeta}(\Gp_J)^*$.
By \eqref{eq:C-tri} and \eqref{eq:tri-p1} we have
\begin{equation}
\label{eq:C-tri-p}
\tilde{U}_{K,\zeta}(\Gn)^\bigstar
\otimes
O_{K,\zeta}(H)
\otimes
U_{K,\zeta}(\Gn^+_J)^\bigstar
\cong
\hat{O}_{K,\zeta}(P_J).
\end{equation}

Similarly to \eqref{eq:adOG} we define the adjoint action of $U_{K,\zeta}(\Gp_J)$ on $O_{K,\zeta}(P_J)$ by 
\begin{equation}
\label{eq:adOP}
\ad(z)(\varphi)=\sum_{(z)}z_{(0)}\varphi(S^{-1}z_{(1)})
\qquad
(z\in U_{K,\zeta}(\Gp_J),\; \varphi\in O_{K,\zeta}(P_J)).
\end{equation}

\begin{lemma}
\label{lem:adLP}
\begin{itemize}
\item[(i)]
We have
\begin{align*}
\ad(U_{K,\zeta}(\Gl_J))({U}_{K,\zeta}(\Gl_J))
&\subset
{U}_{K,\zeta}(\Gl_J),
\\
\ad(U_{K,\zeta}(\Gl_J))(\tilde{U}_{K,\zeta}(\Gm_J))
&\subset
\tilde{U}_{K,\zeta}(\Gm_J)
\end{align*}
with respect to the adjoint action of $U_{K,\zeta}(\Gp_J)$ on $U_{K,\zeta}(\Gp_J)$.
\item[(ii)]
We have
\begin{align*}
\ad(U_{K,\zeta}(\Gl_J))
(O_{K,\zeta}(L_J))
\subset&
O_{K,\zeta}(L_J),
\\
\ad(U_{K,\zeta}(\Gl_J))(\tilde{U}_{K,\zeta}(\Gm_J)^\bigstar)
\subset&
\tilde{U}_{K,\zeta}(\Gm_J)^\bigstar
\end{align*}
with respect to the adjoint action of $U_{K,\zeta}(\Gp_J)$ on $O_{K,\zeta}(P_J)$.
\item[(iii)]
For $\varphi\in O_{K,\zeta}(L_J)\subset O_{K,\zeta}(P_J)$, 
$\psi\in \tilde{U}_{K,\zeta}(\Gm_J)^\bigstar\subset O_{K,\zeta}(P_J)$ we have
\[
\ad(u)(\psi\varphi)
=\sum_{(u)}(\ad(u_{(1)})(\psi))(\ad(u_{(0)})(\varphi))
\qquad(u\in U_{K,\zeta}(\Gl_J)).
\]
\end{itemize}
\end{lemma}
\begin{proof}
The first formula of (i) is obvious.
To verify the second one it is sufficient to show the corresponding statement over $\BF$.
This can be proved similarly to 
 \cite[Proposition 1.2]{KMT}.
 (i) is proved.
 
Assume 
$\varphi\in O_{K,\zeta}(L_J)$, 
$\psi\in \tilde{U}_{K,\zeta}(\Gm_J)^\bigstar$, 
$u, a\in U_{K,\zeta}(\Gl_J)$, 
$y\in \tilde{U}_{K,\zeta}(\Gm_J)$.
Then we have
\begin{equation}
\label{eq:OP}
\langle\ad(u)(\psi\varphi), ya\rangle
=\sum_{(u)}
\langle
\ad(u_{(1)})(\psi),
y\rangle
\langle
\ad(u_{(0)})(\varphi),
a\rangle.
\end{equation}
Indeed we have
\begin{align*}
&\langle\ad(u)(\psi\varphi), ya\rangle
=
\langle\psi\varphi,
(\ad(S^{-1}u)(ya)\rangle
=
\sum_{(u)}
\langle\psi\varphi,
(\ad(S^{-1}u_{(1)})(y))(\ad(S^{-1}u_{(0})(a))\rangle
\\
=&
\sum_{(u)}
\langle\psi,
\ad(S^{-1}u_{(1)})(y)\rangle
\langle\varphi,
\ad(S^{-1}u_{(0})(a)\rangle
=\sum_{(u)}
\langle
\ad(u_{(1)})(\psi),
y\rangle
\langle
\ad(u_{(0)})(\varphi),
a\rangle
\end{align*}
by (i) and Lemma \ref{lem:OP} (iv).
Setting $\psi=1_{O_{K,\zeta}(P_J)}=\varepsilon$ in \eqref{eq:OP} we obtain
\[
\langle\ad(u)(\varphi), ya\rangle
=
\varepsilon(y)
\langle
\ad(u)(\varphi),
a\rangle.
\]
Hence $\ad(u)(\varphi)\in O_{K,\zeta}(L_J)$.
Similarly, setting $\varphi=1_{O_{K,\zeta}(P_J)}$ in \eqref{eq:OP} we obtain
$\ad(u)(\psi)\in \tilde{U}_{K,\zeta}(\Gm_J)^\bigstar$.
(ii) is proved.

By (ii) and Lemma \ref{lem:OP} (iv)
we also obtain (iii) from \eqref{eq:OP}.
\end{proof}
Applying the natural Hopf algebra homomorphism 
$O_{K,\zeta}(P_J)\to O_{K,\zeta}(B)$ induced by 
$U_{K,\zeta}(\Gb)\hookrightarrow
U_{K,\zeta}(\Gp_J)$ to
Lemma \ref{lem:OP} and Lemma \ref{lem:adLP} we obtain the following results.

\begin{lemma}
\label{lem:OPX}
Regard 
$O_{K,\zeta}(B_J)$, 
${\tilde{U}}_{K,\zeta}(\Gm_J)^\bigstar$
as subspaces of $U_{K,\zeta}(\Gb)^*$ via the embeddings
\begin{align}
\label{eq:embp0X}
O_{K,\zeta}(B_J)&\hookrightarrow U_{K,\zeta}(\Gb)^*
\qquad
(\langle\varphi, ya\rangle=\varepsilon(y)\varphi(a)),
\\
\label{eq:embpmX}
{\tilde{U}}_{K,\zeta}(\Gm_J)^\bigstar&\hookrightarrow U_{K,\zeta}(\Gb)^*
\qquad
(\langle\psi, ya\rangle=\psi(y)\varepsilon(a))
\end{align}
for $y\in {\tilde{U}}_{K,\zeta}(\Gm_J)$, $a\in {U}_{K,\zeta}(\Gb_J)$ (see \eqref{eq:trig-p}).
\begin{itemize}
\item[(i)]
The multiplication of $U_{K,\zeta}(\Gb)^*$ induces the isomorphism
\begin{equation}
\label{eq:embPPX}
{\tilde{U}}_{K,\zeta}(\Gm_J)^\bigstar
\otimes
O_{K,\zeta}(B_J)
\cong
O_{K,\zeta}(B)
\subset
U_{K,\zeta}(\Gb)^*
\end{equation}
of $K$-modules.
\item[(ii)]
\eqref{eq:embp0X} is a homomorphism of $K$-algebras.
\item[(iii)]
\eqref{eq:embpmX}  is a homomorphism of $K$-algebras, where ${\tilde{U}}_{K,\zeta}(\Gm_J)^\bigstar$ is regarded as a $K$-algebra via 
Lemma \ref{lem:OP} (iii).
\item[(iv)]
For $\varphi\in O_{K,\zeta}(B_J)$, 
$\psi\in
{\tilde{U}}_{K,\zeta}(\Gm_J)^\bigstar$ we have
\[
\langle\psi\varphi,ya\rangle
=
\langle\psi,y\rangle
\langle\varphi,a\rangle
\qquad(
y\in {\tilde{U}}_{K,\zeta}(\Gm_J),\;
a\in {U}_{K,\zeta}(\Gb_J))
\]
with respect to the multiplication of ${U}_{K,\zeta}(\Gb)^*$.
\end{itemize}
\end{lemma}

\begin{lemma}
\label{lem:adLPX}
\begin{itemize}
\item[(i)]
We have
\begin{align*}
\ad(U_{K,\zeta}(\Gb_J))({U}_{K,\zeta}(\Gb_J))
&\subset
{U}_{K,\zeta}(\Gb_J),
\\
\ad(U_{K,\zeta}(\Gb_J))(\tilde{U}_{K,\zeta}(\Gm_J))
&\subset
\tilde{U}_{K,\zeta}(\Gm_J)
\end{align*}
with respect to the adjoint action of $U_{K,\zeta}(\Gb)$ on $U_{K,\zeta}(\Gb)$ 
\item[(ii)]
We have
\begin{align*}
\ad(U_{K,\zeta}(\Gb_J))
(O_{K,\zeta}(B_J))
\subset&
O_{K,\zeta}(B_J),
\\
\ad(U_{K,\zeta}(\Gb_J))(\tilde{U}_{K,\zeta}(\Gm_J)^\bigstar)
\subset&
\tilde{U}_{K,\zeta}(\Gm_J)^\bigstar
\end{align*}
with respect to the adjoint action of $U_{K,\zeta}(\Gb)$ on $O_{K,\zeta}(B)$
\item[(iii)]
For $\varphi\in O_{K,\zeta}(B_J)\subset O_{K,\zeta}(B)$, 
$\psi\in \tilde{U}_{K,\zeta}(\Gm_J)^\bigstar\subset O_{K,\zeta}(B)$ we have
\[
\ad(u)(\psi\varphi)
=\sum_{(u)}(\ad(u_{(1)})(\psi))(\ad(u_{(0)})(\varphi))
\qquad(u\in U_{K,\zeta}(\Gb_J)).
\]
\end{itemize}
\end{lemma}
\section{Induction functor}
\subsection{}
We fix  $\rho\in\Gh^*$ satisfying
$\langle\rho,\alpha_i^\vee\rangle=1$ for any $i\in I$. 
We  define a shifted action 
$W\times \Gh^*\to \Gh^*$ ($(w,\lambda)\mapsto w\circ\lambda$) 
of $W$ on $\Gh^*$ by
\[
w\circ\lambda
=w(\lambda+\rho)-\rho
\qquad(w\in W,\; \lambda\in \Gh^*).
\]
This action does not depend on the choice of $\rho$ and preserves $Q$ and $\Lambda$.

We set
\begin{equation}
h_G=
\max_{\alpha\in\Delta^+}\langle\rho,\alpha^\vee\rangle+1.
\end{equation}
This is the largest number among the Coxeter numbers of the irreducible components of $\Delta$.
For $X\subset \Delta^+$ we set $\lambda_X=-\sum_{\alpha\in X}\alpha$.
It is well-known that
\begin{equation}
|\langle\lambda_X+\rho,\alpha^\vee\rangle|\leqq h_G-1
\qquad(\alpha\in\Delta)
\end{equation}
(see, for example, \cite[Section 2.2]{AJ}).
By Kostant \cite[Section 5.10]{Kos} we have the following.
\begin{lemma}
\label{lem:Kos}
\begin{itemize}
\item[(i)]
For $X\subset\Delta^+$
we have $W\circ\lambda_X\cap \Lambda^+\ne\emptyset$ if and only if  $W\circ\lambda_X\cap \Lambda^+=\{0\}$.
\item[(ii)]
Denote by $\CX$ the set of $X\subset\Delta^+$
satisfying 
$W\circ\lambda_X\cap \Lambda^+=\{0\}$.
Then we have a bijection $W\to \CX$ given by
$w\mapsto \Delta^+\cap (-w^{-1}\Delta^+)$.
\end{itemize}
\end{lemma}
\subsection{}
For $J_1\subset J_2\subset I$ 
we define a left exact functor
\begin{equation}
\Ind^{\Gp_{J_2},\Gp_{J_1}}:
\Mod_\inte(U_{K,\zeta}(\Gp_{J_1}))\to
\Mod_\inte(U_{K,\zeta}(\Gp_{J_2}))
\end{equation}
as follows.
Let $M\in \Mod_\inte(U_{K,\zeta}(\Gp_{J_1}))$.
We consider 
$\Hom_{U_{K,\zeta}(\Gp_{J_1})}
(U_{K,\zeta}(\Gp_{J_2}),M)$, 
where $U_{K,\zeta}(\Gp_{J_2})$ is regarded as a left 
$U_{K,\zeta}(\Gp_{J_1})$-module by the left multiplication.
It is a left $U_{K,\zeta}(\Gp_{J_2})$-module by
\[
(u\cdot \theta)(u')=\theta(u'u)
\qquad
(u, u'\in U_{K,\zeta}(\Gp_{J_2}), 
\theta\in \Hom_{U_{K,\zeta}(\Gp_{J_1})}
(U_{K,\zeta}(\Gp_{J_2}),M)).
\]
Then we define $\Ind^{\Gp_{J_2},\Gp_{J_1}}(M)$ to be the largest integrable 
$U_{K,\zeta}(\Gp_{J_2})$-submodule of 
$
\Hom_{U_{K,\zeta}(\Gp_{J_1})}
(U_{K,\zeta}(\Gp_{J_2}),M)
$.
For $i\geqq0$ we denote by
\begin{equation}
R^i\Ind^{\Gp_{J_2},\Gp_{J_1}}:
\Mod_\inte(U_{K,\zeta}(\Gp_{J_1}))\to
\Mod_\inte(U_{K,\zeta}(\Gp_{J_2}))
\end{equation}
the $i$-th derived functor of 
$\Ind^{\Gp_{J_2},\Gp_{J_1}}$.
\begin{remark}
For $M\in \Mod_\inte(U_{K,\zeta}(\Gp_{J_1}))$
define two left $U_{K,\zeta}(\Gp_{J_1})$-module structures of $O_{K,\zeta}(P_{J_2})\otimes M$ by
\begin{align*}
y\bullet(\varphi\otimes m)
=&
\sum_{(y)}\varphi\cdot(S^{-1}y_{(0)})\otimes y_{(1)}m,
\\
y\bullet'(\varphi\otimes m)
=&
\sum_{(y)}\varphi\cdot(Sy_{(1)})\otimes y_{(0)}m
\end{align*}
for $y\in U_{K,\zeta}(\Gp_{J_1})$, $\varphi\in O_{K,\zeta}(P_{J_2})$, $m\in M$.
Identifying $O_{K,\zeta}(P_{J_2})\otimes M$  with a subspace of 
$\Hom_{K}
(U_{K,\zeta}(\Gp_{J_2}),M)$ in a natural manner the following three conditions on $\theta\in\Hom_{K}
(U_{K,\zeta}(\Gp_{J_2}),M)$ are equivalent:
\begin{itemize}
\item[(a)]
$\theta\in\Ind^{\Gp_{J_2},\Gp_{J_1}}(M)$,
\item[(b)]
$\theta\in O_{K,\zeta}(P_{J_2})\otimes M$, and $y\bullet \theta=\varepsilon(y)\theta$ for any $y\in U_{K,\zeta}(\Gp_{J_1})$,
\item[(c)]
$\theta\in O_{K,\zeta}(P_{J_2})\otimes M$, and $y\bullet' \theta=\varepsilon(y)\theta$ for any $y\in U_{K,\zeta}(\Gp_{J_1})$.
\end{itemize}
\end{remark}

We have the transitivity
\begin{equation}
\label{eq:Ind-trans}
\Ind^{\Gp_{J_3},\Gp_{J_2}}
\circ
\Ind^{\Gp_{J_2},\Gp_{J_1}}
=
\Ind^{\Gp_{J_3},\Gp_{J_1}}
\qquad(J_1\subset J_2\subset J_3\subset I).
\end{equation}

We will also use the following standard result (see \cite{APW}).
\begin{lemma}
\label{lem:TI}
Let $J_1\subset J_2\subset I$.
For $V\in\Mod_\inte(U_{K,\zeta}(\Gp_{J_2}))$ and $M\in\Mod_\inte(U_{K,\zeta}(\Gp_{J_1}))$
we have canonical isomorphisms
\begin{align*}
V{\otimes}R^i\Ind^{\Gp_{J_2},\Gp_{J_1}}(M)
\cong&
R^i\Ind^{\Gp_{J_2},\Gp_{J_1}}(V{\otimes}M),
\\
R^i\Ind^{\Gp_{J_2},\Gp_{J_1}}(M)\otimes V
\cong&
R^i\Ind^{\Gp_{J_2},\Gp_{J_1}}(M{\otimes}V)
\end{align*}
of $U_{K,\zeta}(\Gp_{J_2})$-modules 
for any $i$.
In the case $i=0$ the isomorphisms are given by
\begin{align*}
V{\otimes}\Ind^{\Gp_{J_2},\Gp_{J_1}}(M)
\cong&
\Ind^{\Gp_{J_2},\Gp_{J_1}}(V{\otimes}M)
\quad&
(v\otimes\theta\mapsto[u\mapsto
\sum_{(u)}u_{(0)}v\otimes\theta(u_{(1)})]),
\\
\Ind^{\Gp_{J_2},\Gp_{J_1}}(M)\otimes V
\cong&
\Ind^{\Gp_{J_2},\Gp_{J_1}}(M{\otimes}V)
\quad&
(\theta\otimes v\mapsto
[u\mapsto
\sum_{(u)}\theta(u_{(0)})\otimes u_{(1)}v]).
\end{align*}
\end{lemma}

\subsection{}
For $\lambda\in\Lambda$ 
we extend $\chi_\lambda:U_{K,\zeta}(\Gh)\to K$ to the algebra homomorphism
$\chi_\lambda:U_{K,\zeta}(\Gb)\to K$ by setting $\chi_\lambda(y)=\varepsilon(y)$ for $y\in \tilde{U}_{K,\zeta}(\Gn)$.
We denote by $K_\lambda=K1_\lambda$ the one-dimensional $U_{K,\zeta}(\Gb)$-module given by 
$y1_\lambda=\chi_\lambda(y)1_\lambda$.
It is well-known that for $\lambda\in\Lambda$ we have
\begin{equation}
\label{eq:IndOne}
\Ind^{\Gg,\Gb}(K_\lambda)
\cong
\begin{cases}
V^*_{K,\zeta}(\lambda)
\qquad&(\lambda\in\Lambda^+)
\\
0&(\lambda\notin\Lambda^+).
\end{cases}
\end{equation}
We have also the following analogue of the Borel-Weil-Bott theorem.
\begin{proposition}[\cite{APW}]
\label{prop:BBW0}
Assume that $\zeta\in K^\times$ is not a root of unity.
For $\lambda\in\Lambda$ we have the following.
\begin{itemize}
\item[(i)]
If $\langle\lambda+\rho,\alpha^\vee\rangle=0$ for some $\alpha\in\Delta$, then 
we have
$
R^i\Ind^{\Gg,\Gb}(K_\lambda)=0
$
for any $i$.
\item[(ii)]
Assume $\langle\lambda+\rho,\alpha^\vee\rangle\ne0$ for any $\alpha\in\Delta$.
For $w\in W$ such that $w\circ\lambda\in\Lambda^+$ we have
\[
R^i\Ind^{\Gg,\Gb}(K_\lambda)\cong
\begin{cases}
V^*_{K,\zeta}(w\circ\lambda)
\qquad&(i=\ell(w))
\\
0&(i\ne\ell(w)).
\end{cases}
\]
\end{itemize}
\end{proposition}
However, the behavior  of $R^i\Ind(K_\lambda)$ for   $i>0$ is not well-understood when $\zeta$ is a root of unity.
In this paper we will use the following result.\begin{proposition}[\cite{A}]
\label{prop:BBW}
Assume that $\zeta\in K^\times$ is a root of unity.
For $\alpha\in\Delta$ we denote by $\ell_\alpha$ the multiplicative order of $\zeta^{(\alpha,\alpha)}$.
Then for
$\lambda\in\Lambda$ satisfying
\[
|\langle\lambda+\rho,\alpha^\vee\rangle|
\leqq \ell_\alpha\qquad(\forall\alpha\in\Delta)
\]
 we have the following.
\begin{itemize}
\item[(i)]
If $\langle\lambda+\rho,\alpha^\vee\rangle=0$ for some $\alpha\in\Delta$, then 
we have
$
R^i\Ind^{\Gg,\Gb}(K_\lambda)=0
$
for any $i$.
\item[(ii)]
Assume $\langle\lambda+\rho,\alpha^\vee\rangle\ne0$ for any $\alpha\in\Delta$.
For $w\in W$ such that $w\circ\lambda\in\Lambda^+$ we have
\[
R^i\Ind^{\Gg,\Gb}(K_\lambda)\cong
\begin{cases}
V^*_{K,\zeta}(w\circ\lambda)
\qquad&(i=\ell(w))
\\
0&(i\ne\ell(w)).
\end{cases}
\]
\end{itemize}
\end{proposition}
\subsection{}
Let $J\subset I$.
We have the induction functor
\begin{align*}
\Ind^{\Gl_J,\Gb_J}:
\Mod_\inte(U_{K,\zeta}(\Gb_J))
\to
\Mod_\inte(U_{K,\zeta}(\Gl_J))
\end{align*}
similarly to 
 $\Ind^{\Gg,\Gb}$
 (see \eqref{eq:CRJ}).

We will use the following fact later.
\begin{lemma}
\label{lem:IndPL}
We have the following commutative diagram of functors:
\[
\xymatrix@C=50pt{
\Mod_\inte(U_{K,\zeta}(\Gb))
\ar[d]
\ar[r]^{R^i\Ind^{\Gp_J,\Gb}}
&
\Mod_\inte(U_{K,\zeta}(\Gp_J))
\ar[d]
\\
\Mod_\inte(U_{K,\zeta}(\Gb_J))
\ar[r]^{R^i\Ind^{\Gl_J,\Gb_J}}
&
\Mod_\inte(U_{K,\zeta}(\Gl_J)),
}
\]
where the vertical arrows are given by the restriction.
\end{lemma}
\begin{proof}
It is sufficient to show the statement for $i=0$.
Let $M\in\Mod_\inte(U_{K,\zeta}(\Gb))$.
We first show 
\begin{equation}
\label{eq:Hom-cong}
\Hom_{U_{K,\zeta}(\Gb)}(U_{K,\zeta}(\Gp_J),M)
\cong
\Hom_{U_{K,\zeta}(\Gb_J)}(U_{K,\zeta}(\Gl_J),M).
\end{equation}
By
\[
U_{K,\zeta}(\Gp_J)\cong U_{K,\zeta}(\Gb)\otimes U_{K,\zeta}(\Gn_J^+),
\qquad
U_{K,\zeta}(\Gl_J)\cong U_{K,\zeta}(\Gb_J)\otimes U_{K,\zeta}(\Gn_J^+)
\]
we have 
\[
\Hom_{U_{K,\zeta}(\Gb)}(U_{K,\zeta}(\Gp_J),M)
\cong
\Hom_{K}(U_{K,\zeta}(\Gn_J^+),M)
\cong
\Hom_{U_{K,\zeta}(\Gb_J)}(U_{K,\zeta}(\Gl_J),M).
\]
\eqref{eq:Hom-cong} is verified.
In \eqref{eq:Hom-cong}
$\theta\in \Hom_{U_{K,\zeta}(\Gb_J)}(U_{K,\zeta}(\Gl_J),M)$ corresponds to
\newline
$\tilde{\theta}\in
\Hom_{U_{K,\zeta}(\Gb)}(U_{K,\zeta}(\Gp_J),M)$ satisfying
\[
\tilde{\theta}(ba)=b\theta(a)
\qquad(a\in U_{K,\zeta}(\Gl_J), b\in U_{K,\zeta}(\Gb)).
\]
We need to show that $U_{K,\zeta}(\Gp_J)$-integrable elements in the left side of 
\eqref{eq:Hom-cong} correspondes to 
$U_{K,\zeta}(\Gl_J)$-integrable elements in the right side.
We see easily that the isomorphism \eqref{eq:Hom-cong} respects the action of $U_{K,\zeta}(\Gl_J)$, and hence it is sufficient to show that  for any $U_{K,\zeta}(\Gl_J)$-integrable element $\theta$ of the right side of \eqref{eq:Hom-cong} and for any $i\in I\setminus J$ we have $f_i^{(n)}\tilde{\theta}=0$ for $n\gg0$.
For $a\in U_{K,\zeta}(\Gn_J^+)$, 
$b\in U_{K,\zeta}(\Gb)$
we have
\[
(f_i^{(n)}\tilde{\theta})(ba)=
\tilde{\theta}(baf_i^{(n)})
=
\tilde{\theta}(bf_i^{(n)}a)
=
bf_i^{(n)}({\theta}(a)), 
\]
and hence we have only to verify that
$\theta(U_{K,\zeta}(\Gn_J^+))$ is finite-dimensional.
Since $\theta$ is $U_{K,\zeta}(\Gl_J)$-integrable,
$U_{K,\zeta}(\Gl_J)\theta$ is finite-dimensional, and 
hence
$\theta(U_{K,\zeta}(\Gn_J^+))$ is finite-dimensional by
\[
\theta(U_{K,\zeta}(\Gn_J^+))
\subset\theta(U_{K,\zeta}(\Gl_J))
\subset
(U_{K,\zeta}(\Gl_J)\theta)(1).
\]
\end{proof}
\begin{proposition}
\label{prop:IndM}
For $V\in\Mod_\inte(U_{K,\zeta}(\Gp_J))$ we have
\[
R^i\Ind^{\Gp_J,\Gb}(V)
\cong
\begin{cases}
V\qquad&(i=0)
\\
0&(i\ne0).
\end{cases}
\]
\end{proposition}
\begin{proof}
By Lemma \ref{lem:TI} we may assume that $V$ is the one-dimensional trivial  $U_{K,\zeta}(\Gp_J)$-module $K$ corresponding to the character $\varepsilon:U_{K,\zeta}(\Gp_J)\to K$.
Note that we have a canonical homomorphism 
$K\to \Ind^{\Gp_J,\Gb}(K)$ of $U_{K,\zeta}(\Gp_J)$-modules given by $K\ni 1\mapsto \varepsilon\in\Hom_{U_{K,\zeta}(\Gb)}(U_{K,\zeta}(\Gp_J),K)$.
By Lemma \ref{lem:IndPL} it is sufficient to show that the canonical homomorphism
$K\to \Ind^{\Gl_J,\Gb_J}(K)$ of $U_{K,\zeta}(\Gl_J)$-modules is an isomorphism and that 
$R^i\Ind^{\Gl_J,\Gb_J}(K)=\{0\}$ for $i\ne0$.
This follows from \eqref{eq:IndOne} and Proposition \ref{prop:BBW} applied to the case $\Gg=\Gl_J$.
\end{proof}
\subsection{}
We define a twisted action 
of $W$ on $O_{K,\zeta}(H)
=\bigoplus_{\lambda\in\Lambda}K\chi_\lambda
$ 
by
\[
w\circ\chi_\lambda=
\zeta^{2\langle\rho,\jmath_0(w\lambda-\lambda)\rangle}
\chi_{w\lambda}
\qquad
(w\in W, \lambda\in\Lambda).
\]
We set 
\begin{align*}
O_{K,\zeta}(H)^{W\circ}
=&\{f\in O_{K,\zeta}(H)\mid
w\circ f=f\;\;(w\in W)\},
\\
O_{K,\zeta}(G)^{\ad(U_{K,\zeta}(\Gg))}
=&
\{\varphi\in  O_{K,\zeta}(G)
\mid
\ad(x)(\varphi)=\varepsilon(x)\varphi
\;\;(x\in U_{K,\zeta}(\Gg))\}.
\end{align*}
The natural Hopf algebra  homomorphism 
$O_{K,\zeta}(G)\to O_{K,\zeta}(H)$ corresponding to $U_{K,\zeta}(\Gh)\hookrightarrow U_{K,\zeta}(\Gg)$
induces the isomorphism
\[
 O_{K,\zeta}(G)^{\ad(U_{K,\zeta}(\Gg))}
\cong 
 O_{K,\zeta}(H)^{W\circ}
\]
of $K$-algebras (see \cite[Proposition 6.4]{T3}).
In particular, we have a natural embedding 
$
O_{K,\zeta}(H)^{W\circ}\hookrightarrow
 O_{K,\zeta}(G)
$
of $K$-algebras.
Applying the above argument to $L_J$ instead of $G$ we also obtain a natural embedding 
\begin{equation}
\label{eq:O-embL}
O_{K,\zeta}(H)^{W_J\circ}
\hookrightarrow
 O_{K,\zeta}(L_J)
 \qquad(J\subset I)
 \end {equation}
of $K$-algebras 
whose image coincides with 
$O_{K,\zeta}(L_J)^{\ad(U_{K,\zeta}(\Gl_J))}$.
Composing this with $\pi_J^*:O_{K,\zeta}(L_J)\to O_{K,\zeta}(P_J)$ we also obtain an embedding
\begin{equation}
\label{eq:O-embP}
\iota_J:O_{K,\zeta}(H)^{W_J\circ}
\hookrightarrow
 O_{K,\zeta}(P_J)
 \qquad(J\subset I)
 \end {equation}
of $K$-algebras whose image is contained in 
$O_{K,\zeta}(P_{J})^{\ad(U_{K,\zeta}(\Gp_{J}))}$.

Now let us define 
for $J_1\subset J_2\subset I$ a homomorphism
\begin{equation}
\label{eq:Upsilon-gen}
\Upsilon^{J_2,J_1}:
 O_{K,\zeta}(P_{J_2})_\ad
\otimes_{ O_{K,\zeta}(H)^{W_{J_2}\circ}}
 O_{K,\zeta}(H)^{W_{J_1}\circ}
\to\Ind^{\Gp_{J_2},\Gp_{J_1}}( O_{K,\zeta}(P_{J_1})_\ad)
\end{equation}
of $U_{K,\zeta}(\Gp_{J_2})$-modules.
Here, we added the subscript ``ad'' since $O_{K,\zeta}(P_{J})$ for $J=J_1, J_2$ is regarded as a $U_{K,\zeta}(\Gp_J)$-module with respect to the adjoint action.
We regard 
$O_{K,\zeta}(P_{J_2})_\ad$ as a right
$O_{K,\zeta}(H)^{W_{J_2}\circ}$-module by the right multiplication through $\iota_{J_2}$.
Since the image of $\iota_{J_2}$ is contained in 
$O_{K,\zeta}(P_{J_2})^{\ad(U_{K,\zeta}(\Gp_{J_2}))}$, we have a well defined action of 
$U_{K,\zeta}(\Gp_{J_2})$ on 
$O_{K,\zeta}(P_{J_2})_\ad
\otimes_{O_{K,\zeta}(H)^{W_{J_2}\circ}}
 O_{K,\zeta}(H)^{W_{J_1}\circ}
$ given by 
\[
u\cdot(\varphi\otimes\chi)
=\ad(u)(\varphi)\otimes\chi\
\quad(u\in U_{K,\zeta}(\Gp_{J_2}), 
\;\varphi\in O_{K,\zeta}(P_{J_2}),\;
\chi\in  O_{K,\zeta}(H)^{W_{J_1}\circ}).
\]
We define \eqref{eq:Upsilon-gen} by
\[
(\Upsilon^{J_2,J_1}(\varphi\otimes\chi))(u)
=
\res(\ad(u)(\varphi))\iota_{J_1}(\chi)
\]
for 
$\varphi\in  O_{K,\zeta}(P_{J_2})_\ad$, 
$\chi\in  O_{K,\zeta}(H)^{W_{J_1}\circ}$,
$u\in U_{K,\zeta}(\Gp_{J_2})$,
where $\res:O_{K,\zeta}(P_{J_2})\to O_{K,\zeta}(P_{J_1})$ is the Hopf algebra homomorphism induced by 
$U_{K,\zeta}(\Gp_{J_1})\hookrightarrow U_{K,\zeta}(\Gp_{J_2})$.
\section{Equivariant modules}
\subsection{}
Let $J\subset I$.
We say that $\CO$ is a $U_{K,\zeta}(\Gp_J)$-equivariant $K$-algebra if it is a $K$-algebra endowed with
a $U_{K,\zeta}(\Gp_J)$-bimodule structure satisfying
\[
u(\varphi\psi)=
\sum_{(u)}(u_{(0)}\varphi)(u_{(1)}\psi), 
\qquad
(\varphi\psi)u=
\sum_{(u)}(\varphi u_{(0)})(\psi u_{(1)})
\]
for $u\in U_{K,\zeta}(\Gp_J)$, $\varphi, \psi\in\CO$.
Note that $O_{K,\zeta}(P_{J'})$ for $J\subset J'\subset I$ is a $U_{K,\zeta}(\Gp_J)$-equivariant $K$-algebra.

For a $U_{K,\zeta}(\Gp_J)$-equivariant $K$-algebra $\CO$ 
we denote by $\Mod^\CO(U_{K,\zeta}(\Gp_J))$ the category consisting of an $\CO$-module $M$
 equipped with a $U_{K,\zeta}(\Gp_J)$-module  structure satisfying
\[
u(\varphi m)=\sum_{(u)}(u_{(0)}\varphi(S^{-1}u_{(2)}))
(u_{(1)}m)
\qquad
(u\in U_{K,\zeta}(\Gp_J), \varphi\in\CO, m\in M).
\]
The objects of $\Mod^\CO(U_{K,\zeta}(\Gp_J))$ 
are called 
$U_{K,\zeta}(\Gp_J)$-equivariant $\CO$-modules.
Note that $\CO$ endowed with the $\CO$-module structure given by the left multiplication and the $U_{K,\zeta}(\Gp_J)$-module structure given 
by
\[
U_{K,\zeta}(\Gp_J)\times \CO\to \CO
\qquad
((u,\varphi)\mapsto \sum_{(u)}u_{(0)}\varphi(S^{-1}u_{(1)}))
\]
is an object of $\Mod^\CO(U_{K,\zeta}(\Gp_J))$.

We denote by $\Mod^\CO_\inte(U_{K,\zeta}(\Gp_J))$ 
(resp.\ 
$\Mod^\CO_\diag(U_{K,\zeta}(\Gp_J))$)
the full subcategory of 
$\Mod^\CO(U_{K,\zeta}(\Gp_J))$ 
consisting of 
those modules $M\in \Mod^\CO(U_{K,\zeta}(\Gp_J))$ 
that are integrable 
(resp.\ diagonalizable) as  
$U_{K,\zeta}(\Gp_J)$-modules.
\subsection{}
\label{subsec:tildetensor}
For a left $U_{K,\zeta}(\Gp_J)$-module
$M$ and a $U_{K,\zeta}(\Gp_J)$-bimodule $N$ 
we define a left $U_{K,\zeta}(\Gp_J)$-module structure
of $N\otimes M$ by
\[
u\cdot(n\otimes m)
=\sum_{(u)}u_{(0)}n(S^{-1}u_{(2)})\otimes u_{(1)}m
\qquad
(u\in U_{K,\zeta}(\Gp_J)).
\]
We denote this 
$U_{K,\zeta}(\Gp_J)$-module 
by $N\Breve{\otimes}M$.

For a homomorphism $f:\CO\to\CO'$ of  $U_{K,\zeta}(\Gp_J)$-equivariant $K$-algebras we have a natural functor
\begin{equation}
\label{eq:equiv-ext}
\CO'\Breve{\otimes}_\CO(\bullet):\Mod^\CO(U_{K,\zeta}(\Gp_J))
\to
\Mod^{\CO'}(U_{K,\zeta}(\Gp_J))
\end{equation}
sending $M$ to $\CO'\Breve{\otimes}_\CO M$.

Similarly to Lemma \ref{lem:TI} we have the following.
\begin{lemma}
\label{lem:TTI}
Let $J_1\subset J_2\subset I$.
Let $V$ be a 
$U_{K,\zeta}(\Gp_{J_2})$-bimodule, which is integrable both as a left and a right $U_{K,\zeta}(\Gp_{J_2})$-module, 
and let $M$ be an integrable left 
$U_{K,\zeta}(\Gp_{J_1})$-module.
Then we have a canonical isomorphism
\[
R^i\Ind^{\Gp_{J_2},\Gp_{J_1}}(V\Breve{\otimes}M)
\cong
V
\Breve{\otimes}
R^i\Ind^{\Gp_{J_2},\Gp_{J_1}}(M)
\]
of $U_{K,\zeta}(\Gp_{J_2})$-modules 
for any $i$.
\end{lemma}

Let $J_1\subset J_2\subset I$ and let $\CO$ be a $U_{K,\zeta}(\Gp_{J_2})$-equivariant $K$-algebra.
We assume that $\CO$ is integrable both as a left and a  right 
$U_{K,\zeta}(\Gp_{J_2})$-module.
It is easily seen that 
for $M\in
\Mod_\inte(U_{K,\zeta}(\Gp_{J_1}))$ 
the $U_{K,\zeta}(\Gp_{J_1})$-module
$\CO\Breve{\otimes}M$
belongs to 
$\Mod^\CO_\inte(U_{K,\zeta}(\Gp_{J_1}))$, where 
the left $\CO$-module structure of 
$\CO\Breve{\otimes}M$ is 
given by the  left multiplication on the first factor.

Let $M\in \Mod^\CO_\inte(U_{K,\zeta}(\Gp_{J_1}))$.
Then we have $R^i\Ind^{\Gp_{J_2},\Gp_{J_1}}(M)\in\Mod^\CO_\inte(U_{K,\zeta}(\Gp_{J_2}))$ with respect to the 
$\CO$-module structure of 
$R^i\Ind^{\Gp_{J_2},\Gp_{J_1}}(M)$ given by
\[
\CO\Breve{\otimes}
R^i\Ind^{\Gp_{J_2},\Gp_{J_1}}(M)
\cong
R^i\Ind^{\Gp_{J_2},\Gp_{J_1}}(
\CO\Breve{\otimes}M)
\to
R^i\Ind^{\Gp_{J_2},\Gp_{J_1}}(M),
\]
where the first isomorphism is given by Lemma \ref{lem:TTI} and the second homomorphism is induced by the $\CO$-module structure 
$\CO\otimes M\to M$ of $M$.
Hence the derived induction functor
\[
R^i\Ind^{\Gp_{J_2},\Gp_{J_1}}:
\Mod_\inte(U_{K,\zeta}(\Gp_{J_1}))
\to
\Mod_\inte(U_{K,\zeta}(\Gp_{J_2}))
\]
induces a natural functor
\begin{equation}
R^i\Ind^{\Gp_{J_2},\Gp_{J_1}}:
\Mod^\CO_\inte(U_{K,\zeta}(\Gp_{J_1}))
\to
\Mod^\CO_\inte(U_{K,\zeta}(\Gp_{J_2})).
\end{equation}

\subsection{}
Note that we have an exact functor
\begin{equation}
\label{eq:funhat}
\hat{O}_{K,\zeta}(G)\Breve{\otimes}_{{O}_{K,\zeta}(G)}(\bullet)=\CS^{-1}:
\Mod^{{O}_{K,\zeta}(G)}_{\inte}
(U_{K,\zeta}(\Gb))
\to
\Mod^{\hat{O}_{K,\zeta}(G)}_{\diag}
(U_{K,\zeta}(\Gb))
\end{equation}
sending $M\in
\Mod^{{O}_{K,\zeta}(G)}_{\inte}
(U_{K,\zeta}(\Gb))
$ to
\[
\hat{O}_{K,\zeta}(G)\Breve{\otimes}_{{O}_{K,\zeta}(G)}M
=\CS^{-1}M\in
\Mod^{{\hat{O}_{K,\zeta}(G)}}_{\diag}
(U_{K,\zeta}(\Gb)).
\]
Note also that \eqref{eq:finite}
induces a left exact functor
\begin{equation}
(\bullet)^\inte:
\Mod^{\hat{O}_{K,\zeta}(G)}_{\diag}
(U_{K,\zeta}(\Gb))\to{\Mod^{{O}_{K,\zeta}(G)}_{\inte}
(U_{K,\zeta}(\Gb))}
\end{equation}
sending $M\in
\Mod^{\hat{O}_{K,\zeta}(G)}_{\diag}
(U_{K,\zeta}(\Gb))$ to
\[
M^\inte=\{m\in M\mid
\dim U_{K,\zeta}(\Gb)m<\infty\}\in
\Mod^{{O}_{K,\zeta}(G)}_{\inte}
(U_{K,\zeta}(\Gb)).
\]
\begin{proposition}
\label{prop:lz}
Assume that $\zeta$ is transcendental over the prime field of $K$ or a root of unity.
If $M\in
\Mod^{{O}_{K,\zeta}(G)}_{\inte}
(U_{K,\zeta}(\Gb))$ satisfies $\CS^{-1}M=\{0\}$, then 
$M=\{0\}$.
\end{proposition}
\begin{proof}
For any $m\in M$ there exists some $\lambda\in\Lambda^+$ such that $\chi_\lambda m=0$ by our assumption.
Let $m\in M$ and set $M'=U_{K,\zeta}(\Gb)m\subset M$.
Since $M'$ is finite-dimensional, there exists $\lambda\in\Lambda^+$ such that $\chi_\lambda M'=\{0\}$.
Let us first show 
\begin{equation}
\label{eq:a}
(U_{K,\zeta}(\Gb)\chi_\lambda)M'=\{0\}.
\end{equation}
By 
$U_{K,\zeta}(\Gb)\chi_\lambda=
\tilde{U}_{K,\zeta}(\Gn)\chi_\lambda$
it is sufficient to show 
$({\tilde{U}}_{K,\zeta}(\Gn)\chi_\lambda)M'=\{0\}$.
We show $({\tilde{U}}_{K,\zeta}(\Gn)_{-\gamma}\chi_\lambda)M'=\{0\}$
for any $\gamma\in Q^+$ by induction on $\gamma$.
If $\gamma=0$, we have 
${\tilde{U}}_{K,\zeta}(\Gn)_0\chi_\lambda=K\chi_\lambda$, and hence the assertion is obvious.
Let $\gamma\in Q^+\setminus\{0\}$, 
and let us show
$({\tilde{U}}_{K,\zeta}(\Gn)_{-\gamma}\chi_\lambda)M'=\{0\}$
assuming
$({\tilde{U}}_{K,\zeta}(\Gn)_{-\gamma'}\chi_\lambda)M'=\{0\}$
for any $\gamma'\in (Q^+\cap(\gamma- Q^+))\setminus\{\gamma\}$.
Take $y\in {\tilde{U}}_{K,\zeta}(\Gn)_{-\gamma}$.
By $\chi_\lambda {\tilde{U}}_{K,\zeta}(\Gn)_{-\gamma''}=\{0\}$ for any $\gamma''\in Q^+\setminus\{0\}$
we have
\[
\{0\}=
y(\chi_\lambda M')
=
\sum_{(y)}
(y_{(0)}\chi_\lambda(S^{-1}y_{(2)}))(y_{(1)}M')
=
\sum_{(y)}
(y_{(0)}\chi_\lambda)(y_{(1)}M').
\]
Hence we obtain $(y\chi_\lambda)M'=\{0\}$ from
\[
\Delta(y)\in y\otimes k_{-\gamma}
+\sum_{\gamma'\in (Q^+\cap(\gamma- Q^+))\setminus\{\gamma\}}
{\tilde{U}}_{K,\zeta}(\Gn)_{-\gamma'}\otimes U_{K,\zeta}(\Gb)
\]
and the hypothesis of induction. 
We have verified \eqref{eq:a}.
By $\chi_\lambda=\Phi_{V_{K,\zeta}(\lambda)}(v^*_\lambda\otimes v_\lambda)$ we have
$U_{K,\zeta}(\Gb)\chi_\lambda
=\Phi_{V_{K,\zeta}(\lambda)}(v^*_\lambda\otimes U_{K,\zeta}(\Gb)v_\lambda)
=\Phi_{V_{K,\zeta}(\lambda)}(v^*_\lambda\otimes V_{K,\zeta}(\lambda))
$.
In particular, we have 
$\Phi_{V_{K,\zeta}(\lambda)}(v^*_\lambda\otimes V_{K,\zeta}(\lambda)_{w\lambda})
\subset 
U_{K,\zeta}(\Gb)\chi_\lambda$
for any $w\in W$.
It follows that 
\begin{align*}
m\in& O_{K,\zeta}(G)m=\sum_{w\in W}O_{K,\zeta}(G)\Phi_{V_{K,\zeta}(\lambda)}(v^*_\lambda\otimes V_{K,\zeta}(\lambda)_{w\lambda})
m
\\
\subset&
O_{K,\zeta}(G)(U_{K,\zeta}(\Gb)\chi_\lambda)M'=\{0\}
\end{align*}
by \cite[Corollary 8.8]{TA}.
\end{proof}
\begin{corollary}
\label{cor:FF}
Assume that $\zeta$ is transcendental over the prime field of $K$ or a root of unity.
\begin{itemize}
\item[(i)]
For $M\in
\Mod^{{O}_{K,\zeta}(G)}_{\inte}
(U_{K,\zeta}(\Gb))$ we have $(\CS^{-1}M)^\inte=M$.
In particular $M\to \CS^{-1}M$ is injective.
\item[(ii)]
The functor \eqref{eq:funhat} is fully faithful.
\end{itemize}
\end{corollary}
\begin{proof}
(i) Let $K$ and $L$ be the kernel and the cokernel of $M\to (\CS^{-1}M)^\inte$ respectively.
Applying the exact functor \eqref{eq:funhat} to the exact sequence
\[
0\to K\to M\to (\CS^{-1}M)^\inte\to L\to 0
\]
we obtain
\[
0\to \CS^{-1}K\to \CS^{-1}M\to \CS^{-1}(\CS^{-1}M)^\inte\to \CS^{-1}L\to 0.
\]
Applying the exact functor \eqref{eq:funhat} to $(\CS^{-1}M)^\inte\subset \CS^{-1}M$ we obtain a monomorphism
$\CS^{-1}(\CS^{-1}M)^\inte\to \CS^{-1}M$ 
such that the composite of 
$\CS^{-1}M\to \CS^{-1}(\CS^{-1}M)^\inte\to \CS^{-1}M$ is identity.
Thus
$\CS^{-1}M\to \CS^{-1}(\CS^{-1}M)^\inte$ is an isomorphism.
Hence we have $\CS^{-1}K=\CS^{-1}L=\{0\}$. 
It follows that $K=L=\{0\}$ by Proposition \ref{prop:lz}.
Therefore, $M\to (\CS^{-1}M)^\inte$ is an isomorphism.

(ii) By (i) 
\[
\Hom_{\Mod^{{O}_{K,\zeta}(G)}_{\inte}
(U_{K,\zeta}(\Gb))}(M,N)
\to
\Hom_{\Mod^{\hat{O}_{K,\zeta}(G)}_{\diag}
(U_{K,\zeta}(\Gb))}(\CS^{-1}M,\CS^{-1}N)
\]
is a bijection whose inverse is given by 
$\psi\mapsto\psi|_M$.
 \end{proof}

\section{Quantum groups of type $A$}
\label{sec:A}
\subsection{}
In the rest of this paper we assume that 
the derived group $[G,G]$ of $G$ is isomorphic to $SL_n(\BC)$.
Namely, we consider the situation where $I=\{1,\dots, n-1\}$ with 
\[
a_{ij}=\langle\alpha_j,\alpha_i^\vee\rangle
=
\begin{cases}
2\qquad&(i=j),
\\
-1&(|i-j|=1),
\\
0&(|i-j|>1),
\end{cases}
\]
and the natural restriction map
\[
\Lambda\to
\Hom_\BZ(Q^\vee,\BZ)
\]
is surjective.
In particular there exists $\epsilon_1\in \Lambda$ such that 
$\langle\epsilon_1,\alpha_i^\vee\rangle=\delta_{i,1}$ for any $i\in I$.
We fix such $\epsilon_1$ and set
$\epsilon_r=\epsilon_1-(\alpha_1+\dots\alpha_{r-1})$ for
$r=2,\dots, n$.
For $1\leqq r< s\leqq n$ we set 
$
\alpha_{rs}=\epsilon_r-\epsilon_s
$.
Then we have $\alpha_i=\alpha_{i,i+1}$ and the set of positive roots is given by 
$\Delta^+=\{
\alpha_{rs}\mid1\leqq r< s\leqq n\}$.
The Weyl group $W$ is naturally identified with the symmetric group $\GS_n$.
We have 
$
w\epsilon_r
=
\epsilon_{w(r)}
$ for $w\in \GS_n$ and $1\leqq r\leqq n$.

The $W$-invariant symmetric bilinear form \eqref{eq:sb} is chosen so that $(\alpha,\alpha)=2$ for any $\alpha\in\Delta$.
\subsection{}
Set $V=V_{K,\zeta}(\epsilon_1)$.
We have
$V=\bigoplus_{r=1}^nK v_r$ with
\begin{align*}
h v_r=\chi_{\epsilon_r}(h)v_r,
\qquad
e_i v_r=\delta_{r,i+1}v_{r-1},
\qquad
f_i v_r=\delta_{r,i}v_{r+1}
\end{align*}
for $r=1,\dots, n$, $h\in U_{K,\zeta}(\Gh)$, $i\in I$.
Denote by $\{v^*_r\mid 1\leqq r\leqq n\}$ the basis of $V^*=V^*_{K,\zeta}(\epsilon_1)$ dual to $\{v_r\mid 1\leqq r\leqq n\}$.
We have
\begin{align*}
v^*_rh=\chi_{\epsilon_r}(h)v_r,
\qquad
v^*_re_i=\delta_{r,i}v^*_{r+1},
\qquad
v^*_rf_i=\delta_{r,i+1}v^*_{r-1}
\end{align*}
for $r=1,\dots, n$, $h\in U_{K,\zeta}(\Gh)$, $i\in I$.

For $1\leqq r,s\leqq n$ we set
\[
\xi_{rs}=\Phi_{V}(v^*_r\otimes v_s)
\in O_{K,\zeta}(G).
\]
It is well-known and straightforward to show that they satisfy the relations
\begin{align}
\label{eq:relation1}
\xi_{rs}\xi_{rs'}=&\zeta\xi_{rs'}\xi_{rs}
&(s<s'),
\\
\label{eq:relation2}
\xi_{rs}\xi_{r's}=&\zeta\xi_{r's}\xi_{rs}
&(r<r'),
\\
\label{eq:relation3}
\xi_{rs}\xi_{r's'}=&\xi_{r's'}\xi_{rs}
&(r<r', s>s'),
\\
\label{eq:relation4}
\xi_{rs}\xi_{r's'}=&\xi_{r's'}\xi_{rs}
+(\zeta-\zeta^{-1})\xi_{r's}\xi_{rs'}
&(r<r', s<s').
\end{align}

Let 
$\varkappa:O_{K,\zeta}(G)\to U_{K,\zeta}(\Gn^+)^\bigstar$ be the linear map 
sending $\varphi\in O_{K,\zeta}(G)$ to $\varphi|_{U_{K,\zeta}(\Gn^+)}\in U_{K,\zeta}(\Gn^+)^\bigstar$.
By \eqref{eq:mult-n2} it is a homomorphism of $K$-algebras with respect to the $K$-algebra structure of 
$U_{K,\zeta}(\Gn^+)^\bigstar$ 
given by Lemma \ref{lem:OG} (iii).
We see easily that 
\[
\varkappa(\xi_{rr})=1_{U_{K,\zeta}(\Gn^+)^\bigstar}
=\varepsilon
\;\;(r=1,\dots, n),
\qquad
\varkappa(\xi_{rs})=0\;\;(r>s).
\]
Set $\txi_{rs}=\varkappa(\xi_{rs})
=(U_{K,\zeta}(\Gn^+)_{\alpha_{rs}})^*
\subset U_{K,\zeta}(\Gn^+)^\bigstar$ for $r<s$.
We have
\begin{align}
\label{eq:trel1}
\txi_{rs}\txi_{rs'}=&
\zeta\txi_{rs'}\txi_{rs}
&(r<s<s'),
\\
\label{eq:trel2}
\txi_{rs}\txi_{r's}=&
\zeta\txi_{r's}\txi_{rs}
&(r<r'<s),
\\
\label{eq:trel3}
\txi_{rs}\txi_{r's'}=&
\txi_{r's'}\txi_{rs}
&(r<r'<s'<s),
\\
\txi_{rs}\txi_{r's'}=&
\txi_{r's'}\txi_{rs}
&(r<s<r'<s'),
\\
\label{eq:trel4}
\txi_{rs}\txi_{ss'}=&
\txi_{ss'}\txi_{rs}
+(\zeta-\zeta^{-1})\txi_{rs'}
&(r<s<s'),
\\
\label{eq:trel5}
\txi_{rs}\txi_{r's'}=&
\txi_{r's'}\txi_{rs}
+(\zeta-\zeta^{-1})\txi_{r's}\txi_{rs'}
&(r<r'<s<s')
\end{align}
by 
\eqref{eq:relation1}, \dots \eqref{eq:relation4}.
\begin{lemma}
\label{lem:OM2}
The  ordered monomials:
\begin{equation}
\label{eq:om2}
\txi^{\Ba}
=
\txi_{n-1,n}^{a_{n-1,n}}
(\txi_{n-2,n-1}^{a_{n-2,n-1}}
\txi_{n-2,n}^{a_{n-2,n}})
(\txi_{n-3,n-2}^{a_{n-3,n-2}}
\txi_{n-3,n-1}^{a_{n-3,n-1}}
\txi_{n-3,n}^{a_{n-3,n}}
)
\cdots
(
\txi_{12}^{a_{12}}\cdots
\txi_{1n}^{a_{1n}}
)\end{equation}
for $\Ba=(a_{rs})_{r<s}\in\BZ_{\geqq0}^{n(n-1)/2}$ form a basis of $U_{K,\zeta}(\Gn^+)^\bigstar$.
\end{lemma}
\begin{proof}
By Lemma \ref{lem:duality} the algebra $U_{K,\zeta}(\Gn^+)^\bigstar$ is generated by the one-dimensional subspaces 
$(U_{K,\zeta}(\Gn^+)_{\alpha_i})^*
=K{\txi}_{i,i+1}$ for $i\in I$.
Hence it is also generated by $\txi_{rs}$ for $r<s$.
Then we see easily by \eqref{eq:trel1},\dots,\eqref{eq:trel5} that the ordered monomials \eqref{eq:om2} span $U_{K,\zeta}(\Gn^+)^\bigstar$.
The linearly independence of 
\eqref{eq:om2} follows from the PBW type theorem which implies that for $\beta\in Q^+$ the dimension of 
$(U_{K,\zeta}(\Gn^+)_{\beta})^*$ coincides with the number of the ordered monomial \eqref{eq:om2} contained in $(U_{K,\zeta}(\Gn^+)_{\beta})^*$.
\end{proof}

\subsection{}
We regard 
the tensor algebra 
\[
TV=\bigoplus_{n=0}^\infty V^{\otimes n}
\]
of $V$ as a  $U_{K,\zeta}(\Gg)$-module through the iterated comultiplication.
We define $K$-algebras
$S_{\zeta}V$ and $\wedge_{\zeta} V$ as the quotient
\begin{align*}
S_\zeta V=&TV/(v_r\otimes v_s-\zeta v_s\otimes v_r\mid r<s),
\\
\wedge_\zeta V=
&TV/(v_r\otimes v_r, \;v_r\otimes v_s+\zeta^{-1} v_s\otimes v_r\mid r<s)
\end{align*}
of $TV$.
The multiplications of $S_\zeta V$ and $\wedge_\zeta V$ are denoted by
$(a,b)\mapsto ab$ 
and
$(a,b)\mapsto a\wedge b$ 
respectively
in the following.
We have a basis
\begin{gather*}
\{v_{r_1}v_{r_2}\cdots v_{r_p}
\mid
p\geqq0,\;
1\leqq r_1\leqq\cdots\leqq r_p\leqq n\},
\\
(\text{resp.}\;
\{v_{r_1}\wedge v_{r_2}\wedge\cdots \wedge v_{r_p}
\mid
0\leqq p\leqq n,\;
1\leqq r_1<\cdots< r_p\leqq n\}
)
\end{gather*}
of 
$S_\zeta V$ 
(resp.\ $\wedge_\zeta V$)
We have also  grading
\[
S_\zeta V
=\bigoplus_{p=0}^\infty S_\zeta^pV,
\qquad
\wedge_\zeta V
=\bigoplus_{p=0}^n \wedge_\zeta^pV
\]
by degrees.
It is easily seen that the $U_{K,\zeta}(\Gg)$-module structure of $TV$ induces those of 
$S_\zeta V$ and $\wedge_\zeta V$.

We define a linear map
\begin{equation}
\label{eq:deru}
\deru_p:S_\zeta V\otimes
\wedge_\zeta^{p+1}V
\to
S_\zeta V
\otimes
\wedge_\zeta^{p}V
\qquad(p=0,\dots, n-1)
\end{equation}
by
\begin{align*}
\deru_p(
d\otimes
v_{r_{1}}\wedge\cdots\wedge v_{r_{p+1}}
)
=&
\sum_{a=1}^{p+1}
(-\zeta)^{-(a-1)}
d v_{r_a}
\otimes
v_{r_1}\wedge\cdots\wedge
\widehat{v}_{r_a}
\wedge\cdots\wedge
v_{r_{p+1}}
\\
&\qquad\qquad
(r_1<\dots<r_{p+1}).
\end{align*}
We further define a $K$-algebra homomorphism
\begin{equation}
\label{eq:deru0}
\epsilon:S_\zeta V\to K
\end{equation}
by
\[
\epsilon(S_\zeta^pV)=0\qquad(p>0).
\]
By a straightforward calculation we see that $\deru_p$ and $\epsilon$ are homomorphisms of $U_{K,\zeta}(\Gg)$-modules.
Moreover, by a general theory of quadratic algebras we have the following result 
(see \cite{M2}, \cite{HH}).
\begin{proposition}
\label{prop:Koszul}
The sequence 
\[
0
\to 
S_\zeta V\otimes
\wedge_\zeta^nV
\xrightarrow{\deru_{n-1}}
S_\zeta V\otimes
\wedge_\zeta^{n-1}V
\xrightarrow{\deru_{n-2}}\cdots
\xrightarrow{\deru_{0}}
S_\zeta V\otimes
\wedge_\zeta^0V
\xrightarrow{\epsilon}
K\to0
\]
is exact.
\end{proposition}
Hence the complex
$S_\zeta V\otimes\wedge_\zeta^\bullet V$
gives a resolution of $K$ 
both as an $S_\zeta V$-module and a 
$U_{K,\zeta}(\Gg)$-module.
Here, the action of $S_\zeta V$ and $U_{K,\zeta}(\Gg)$ on $K=K1$ is given by
\[
(S_\zeta^pV)1=0\quad(p>0),
\qquad
u1=\varepsilon(u)1\quad(u\in U_{K,\zeta}(\Gg)).
\]

We call 
$S_\zeta V\otimes\wedge_\zeta^\bullet V$
the Koszul complex for $U_{K,\zeta}(\Gg)$.

\section{Main result}
\label{sec:main}
\subsection{}
We continue to assume that the derived group $[G,G]$ of $G$ is isomorphic to $SL_n(\BC)$.
We will use the notation in Section \ref{sec:A}.
We pose the following assumption in the rest of this paper.
\begin{assumption}
The parameter $\zeta\in K^\times$  is transcendental over the prime field of $K$, or the multiplicative order $\ell$ of $\zeta^2$ is finite and satisfies
$\ell\geqq n$.
\end{assumption}
The aim of this paper is to show the following.
\begin{theorem}
\label{thm:main}
We have 
$R^i\Ind^{\Gg,\Gb}(O_{K,\zeta}(B)_\ad)=\{0\}$ for any $i\ne0$, and 
\[
\Upsilon^{I,\emptyset}:
O_{K,\zeta}(G)_\ad\otimes_{O_{K,\zeta}(H)^{W\circ}}O_{K,\zeta}(H)
\to
\Ind^{\Gg,\Gb}(O_{K,\zeta}(B)_\ad)
\]
is an isomorphism of $U_{K,\zeta}(\Gg)$-modules.
\end{theorem}
Theorem \ref{thm:main} is equivalent to Theorem \ref{thm:intro} in Introduction by Proposition \ref{prop:comod-int}.

\subsection{}
We are going to show Theorem \ref{thm:main} by induction on $n$.
In the case $n=1$ the assertion is obvious since 
the derived group of $G$ is isomorphic to $SL_1(\BC)=\{1\}$.
We will show 
Theorem \ref{thm:main} for $n$ assuming that for $n-1$.
Set 
\[
J=\{2,\dots, n-1\}\subset I.\]
Let us first confirm that the hypothesis of induction implies the following.
\begin{proposition}
\label{prop:hyp}
We have 
$R^i\Ind^{\Gp_J,\Gb}(O_{K,\zeta}(B)_\ad)=\{0\}$ for any $i\ne0$, and 
\[
\Upsilon^{J,\emptyset}:
O_{K,\zeta}(P_J)_\ad\otimes_{O_{K,\zeta}(H)^{W_J\circ}}O_{K,\zeta}(H)
\to
\Ind^{\Gp_J,\Gb}(O_{K,\zeta}(B)_\ad)
\]
is an isomorphism of $U_{K,\zeta}(\Gp_J)$-modules.
\end{proposition}
\begin{proof} We have isomorphisms
\begin{align*}
&
R^i\Ind^{\Gp_J,\Gb}(O_{K,\zeta}(B)_\ad)
\\
\cong&
R^i\Ind^{\Gl_J,\Gb_J}(O_{K,\zeta}(B)_\ad)
\quad&(\text{Lemma \ref{lem:IndPL}})
\\
\cong&
R^i\Ind^{\Gl_J,\Gb_J}(
O_{K,\zeta}(B_J)_\ad\otimes
\tilde{U}_{K,\zeta}(\Gm_J)^\bigstar)
&(\text{Lemma \ref{lem:OPX}, Lemma \ref{lem:adLPX}})
\\
\cong&
R^i\Ind^{\Gl_J,\Gb_J}(O_{K,\zeta}(B_J)_\ad)
\otimes
\tilde{U}_{K,\zeta}(\Gm_J)^\bigstar
&(\text{Lemma \ref{lem:TI}})
\end{align*}
of $U_{K,\zeta}(\Gl_J)$-modules.
Note that by Lemma \ref{lem:OPX} and Lemma \ref{lem:adLPX} (iii) 
we have an isomorphism 
\[
O_{K,\zeta}(B_J)_\ad\otimes
\tilde{U}_{K,\zeta}(\Gm_J)^\bigstar
\cong
O_{K,\zeta}(B)_\ad
\qquad(\varphi\otimes\psi\mapsto \psi\varphi)
\]
of $U_{K,\zeta}(\Gb_J)$-modules.
Since the derived group of $L_J$ is isomorphic to $SL_{n-1}(\BC)$ we have 
\[
R^i\Ind^{\Gl_J,\Gb_J}(O_{K,\zeta}(B_J)_\ad)
\cong
\begin{cases}
O_{K,\zeta}(L_J)_\ad\otimes_{O_{K,\zeta}(H)^{W_J\circ}}O_{K,\zeta}(H)
\quad&(i=0)
\\
0\quad&(i>0)
\end{cases}
\]
by the hypothesis of induction.
Hence 
we have
$R^i\Ind^{\Gp_J,\Gb}(O_{K,\zeta}(B)_\ad)=\{0\}$ for $i\ne0$.
Moreover, by
Lemma \ref{lem:OP}, Lemma \ref{lem:adLP} we obtain 
an isomorphism
\[
\Ind^{\Gp_J,\Gb}(O_{K,\zeta}(B)_\ad)
\cong
O_{K,\zeta}(P_J)_\ad\otimes_{O_{K,\zeta}(H)^{W_J\circ}}O_{K,\zeta}(H)
\]
of $U_{K,\zeta}(\Gl_J)$-modules.
We can easily check that this isomorphism coincides with $\Upsilon^{J,\emptyset}$.
\end{proof}

By Proposition \ref{prop:hyp} and \eqref{eq:Ind-trans} we have
\[
R^i\Ind^{\Gg,\Gb}(O_{K,\zeta}(B)_\ad)
\cong
R^i\Ind^{\Gg,\Gp_J}
(O_{K,\zeta}(P_J)_\ad)\otimes_{O_{K,\zeta}(H)^{W_J\circ}}O_{K,\zeta}(H),
\]
and hence we have only to show that 
$R^i\Ind^{\Gg,\Gp_J}(O_{K,\zeta}(P_J)_\ad)=\{0\}$ for $i>0$, and that
\[
\Upsilon^{I,J}:
O_{K,\zeta}(G)_\ad\otimes_{O_{K,\zeta}(H)^{W\circ}}O_{K,\zeta}(H)^{W_J\circ}
\to
\Ind^{\Gg,\Gp_J}(O_{K,\zeta}(P_J)_\ad)
\]
is an isomorphism.
By Proposition \ref{prop:IndM} we have
\[
R^i\Ind^{\Gp_J,\Gb}(O_{K,\zeta}(P_J)_\ad)
\cong
\begin{cases}
O_{K,\zeta}(P_J)_\ad
\quad&(i=0)
\\
0\quad&(i>0),
\end{cases}
\]
and hence we obtain
\[
R^i\Ind^{\Gg,\Gb}(O_{K,\zeta}(P_J)_\ad)
\cong
R^i\Ind^{\Gg,\Gp_J}(O_{K,\zeta}(P_J)_\ad)
\]
by \eqref{eq:Ind-trans}.
Therefore, the proof is reduced to showing 
\begin{align}
\label{eq:Claim}
R^i\Ind^{\Gg,\Gb}(O_{K,\zeta}(P_J)_\ad)
\cong
\begin{cases}
O_{K,\zeta}(G)_\ad\otimes_{O_{K,\zeta}(H)^{W\circ}}O_{K,\zeta}(H)^{W_J\circ}
\quad&(i=0)
\\
\{0\}
\quad&(i>0),
\end{cases}
\end{align}
where the isomorphism for $i=0$ is given by 
\begin{equation}
\label{eq:oUpsilon}
\overline{\Upsilon}^{I,J}:
O_{K,\zeta}(G)_\ad\otimes_{O_{K,\zeta}(H)^{W\circ}}O_{K,\zeta}(H)^{W_J\circ}
\to
\Ind^{\Gg,\Gb}(O_{K,\zeta}(P_J)_\ad)
\end{equation}
sending $\varphi\otimes\chi$ to 
\[
[U_{K,\zeta}(\Gg)\ni u\mapsto
\res(\ad(u)(\varphi))\iota_J(\chi)
\in O_{K,\zeta}(P_J)]
\in\Ind^{\Gg,\Gb}(O_{K,\zeta}(P_J)_\ad).
\]
Here, $\res:O_{K,\zeta}(G)\to O_{K,\zeta}(P_J)$ is the Hopf algebra homomorphism induced by 
$U_{K,\zeta}(\Gp_J)\hookrightarrow
U_{K,\zeta}(\Gg)$.
\subsection{}
By \eqref{eq:relation1} we have
\begin{equation}
\label{eq:x-com}
\xi_{1r}\xi_{1s}=\zeta\xi_{1s}\xi_{1r}
\qquad(1\leqq r<s).
\end{equation}
We have also 
\begin{equation}
\label{eq:x-tx}
\xi_{11}=\chi_{\epsilon_1},
\qquad
\xi_{1r}=\zeta^{-1}\chi_{\epsilon_1}\txi_{1r}
\qquad(r\geqq2).
\end{equation}
\begin{proposition}
\label{prop:OPJ}
We have 
\[
O_{K,\zeta}(P_J)\cong
O_{K,\zeta}(G)/
\sum_{r=2}^n
O_{K,\zeta}(G)\xi_{1r}.
\]
\end{proposition}
\begin{proof}
It is easily seen that $\xi_{1r}\in\Ker(O_{K,\zeta}(G)\to O_{K,\zeta}(P_J))$ for $r\geqq2$.
Hence we have a surjection
$f:O_{K,\zeta}(G)/
\sum_{r=2}^n
O_{K,\zeta}(G)\xi_{1r}\to
O_{K,\zeta}(P_J)$.
We need to show that $f$ is an isomorphism.
Since $\sum_{r=2}^nK\xi_{1r}$ is $\ad(\Gb)$-invariant, 
$f$ is a morphism in $\Mod^{O_{K,\zeta}(G)}_\inte(U_{K,\zeta}(\Gb))$.
Hence by Corollary \ref{cor:FF}  it is sufficient to show that $\CS^{-1}f$ is an isomorphism.
Namely, we have only to show that the canonical homomorphism
\[
\hat{O}_{K,\zeta}(G)/
\sum_{r=2}^n
\hat{O}_{K,\zeta}(G)\xi_{1r}
\to
\hat{O}_{K,\zeta}(P_J).
\]
is an isomorphism.
By \eqref{eq:x-tx}, 
\eqref{eq:C-tri}, Lemma \ref{lem:OM2}, \eqref{eq:C-tri-p} we have
\begin{align*}
&\hat{O}_{K,\zeta}(G)/
\sum_{r=2}^n
\hat{O}_{K,\zeta}(G)\xi_{1r}
=
\hat{O}_{K,\zeta}(G)/
\sum_{r=2}^n
\hat{O}_{K,\zeta}(G)\txi_{1r}
\\
\cong&
\tilde{U}_{K,\zeta}(\Gn)^\bigstar
\otimes
O_{K,\zeta}(H)
\otimes
\left(
{U}_{K,\zeta}(\Gn^+)^\bigstar
/\sum_{r=2}^n
{U}_{K,\zeta}(\Gn^+)^\bigstar\txi_{1r}\right)
\\
\cong&
\tilde{U}_{K,\zeta}(\Gn)^\bigstar
\otimes
O_{K,\zeta}(H)
\otimes
{U}_{K,\zeta}(\Gn_J^+)^\bigstar
\cong \hat{O}_{K,\zeta}(P_J).
\end{align*}
We are done.
\end{proof}

We denote by $R$ the subalgebra of $O_{K,\zeta}(G)$ generated by $\xi_{11}, \xi_{12},\dots, \xi_{1n}$.
It is a $U_{K,\zeta}(\Gp_J)$-equivariant $K$-subalgebra of 
$O_{K,\zeta}(G)$.
By \eqref{eq:x-tx}, \eqref{eq:com1} and
Lemma \ref{lem:OM2}
the ordered monomials
$
\xi_{11}^{a_{1}}
\xi_{12}^{a_{2}}
\cdots
\xi_{1n}^{a_{n}}
$ for $(a_1,\dots, a_n)\in \BZ_{\geqq0}^n$
form a $K$-basis of $R$.
We denote by 
\[
R=\bigoplus_{p=0}^\infty R^p
\]
the grading given by the degrees of the ordered monomials.
We set
\[
\overline{R}=R/\sum_{r=2}^nR\xi_{1r}
\cong K[\xi_{11}].
\]
It is a  
quotient of $R$ in $\Mod_\inte^R(U_{K,\zeta}(\Gp_J))$.

For $p\in\BZ$ let $\chi^J_{p\epsilon_1}:U_{K,\zeta}(\Gp_J)\to K$ be the character given by
\[
h\mapsto\chi_{p\epsilon_1}(h)
\quad(h\in U_{K,\zeta}(\Gh)), 
\qquad
e_j^{(n)}, f_i^{(n)}\mapsto 0
\quad(i\in I,\; j\in J,\; n\geqq1).
\]
We see easily the following.
\begin{lemma}
\label{lem:R}
As a graded $K$-algebra and a left $U_{K,\zeta}(\Gp_J)$-module, $R$ is isomorphic to $S_\zeta V$ 
via the correspondence
\[
R\ni\xi_{1r}\longleftrightarrow v_r\in S_\zeta V
\qquad(1\leqq r\leqq n).
\]
The right action of $U_{K,\zeta}(\Gp_J)$ on $R$ is given by 
\[
zu=\chi^J_{p\epsilon_1}(u)z
\qquad(z\in R^p, u\in U_{K,\zeta}(\Gp_J)).
\]
\end{lemma}
We write $R_\ad$ (resp.\ $\overline{R}_\ad$) 
fo $R$ (resp.\ $\overline{R}$)
when it
is regarded as  
a $U_{K,\zeta}(\Gp_J)$-module via the adjoint action.
By Lemma \ref{lem:R} the action of $U_{K,\zeta}(\Gp_J)$ on $\overline{R}_\ad$ is trivial.
\begin{lemma}
\label{lem:R-free}
$\hat{O}_{K,\zeta}(G)$ is a flat right $R$-module.
\end{lemma}
\begin{proof}
Denote by $\hat{R}$ the $K$-subalgebra of $\hat{O}_{K,\zeta}(G)$ generated by $R$ and 
$\chi_{\epsilon_1}^{-1}$.
By \eqref{eq:com1}, \eqref{eq:C-tri} and  Lemma \ref{lem:OM2} 
$R$ 
(resp.\ $\hat{R}$)
has a free $K$-basis
consisting of 
$\chi_{\epsilon_1}^m
\txi_{12}^{a_{2}}\dots
\txi_{1n}^{a_{n}}$
for 
$m\in\BZ_{\geqq0}$
(resp.\ $\BZ$), 
$a_{2},\dots, a_{n}\in\BZ_{\geqq0}$.
In particular, $\hat{R}$ is a localization of $R$ 
and hence flat over $R$.
Therefore it is sufficient to show that 
$\hat{O}_{K,\zeta}(G)$ is a free $\hat{R}$-module.
By \eqref{eq:com1}, \eqref{eq:C-tri} and  Lemma \ref{lem:OM2} 
we have only to verify that 
$K[\Lambda]$ is a free $K[e({\epsilon_1})^{\pm1}]$-module (see \ref{subsec:QC1} for the notation).
By $\langle\epsilon_1,\alpha_i^\vee\rangle=1$ we see that $\Lambda/\BZ \epsilon_1$ is torsion free, and hence there exits a $\BZ$-submodule $\Lambda'$ of $\Lambda$ such that $\Lambda=\Lambda'\oplus\BZ \epsilon_1$.
Then we have
$K[\Lambda]\cong K[\Lambda']\otimes K[e({\epsilon_1})^{\pm1}]$.
Hence
$K[\Lambda]$ is free over $K[e({\epsilon_1})^{\pm1}]$.
\end{proof}

\subsection{}
We construct a resolution of $O_{K,\zeta}(P_J)_\ad
\in \Mod^{O_{K,\zeta}(G)}_\inte(U_{K,\zeta}(\Gp_J))$ using the Koszul complex.

Set $V_J=\sum_{r\geqq2}K\xi_{1r}\subset O_{K,\zeta}(G)$ and 
denote by $R_J$ the subalgebra of $O_{K,\zeta}(G)$ generated by $V_J$.
Then $V_J$ and $R_J$ are $U_{K,\zeta}(\Gp_J)$-subbimodules of $O_{K,\zeta}(G)$.
Moreover, $R_J$ is naturally a $U_{K,\zeta}(\Gp_J)$-equivariant $K$-algebra.
The right action of $U_{K,\zeta}(\Gp_J)$ on 
$R_J$ is given by 
\[
\xi_{1r}u=
\chi^J_{\epsilon_1}(u)\xi_{1r}
\qquad(u\in U_{K,\zeta}(\Gp_J), r\geqq2).
\]
Note that as a left $U_{K,\zeta}(\Gl_J)$-module $V_J$ is a highest weight module with highest weight $\epsilon_2$.
Hence applying Proposition \ref{prop:Koszul} to the 
left $U_{K,\zeta}(\Gl_J)$-module $V_J$ we obtain a Koszul resolution 
$S_\zeta V_J\otimes\wedge_\zeta^\bullet V_J$ of 
 $K$ as an $S_\zeta V_J$-module as well as a $U_{K,\zeta}(\Gl_J)$-module.
Since $S_\zeta V_J$ is naturally isomorphic to $R_J$ as a $K$-algebra as well as a left $U_{K,\zeta}(\Gl_J)$-module, 
we replace $S_\zeta V_J$ with $R_J$ in the following and write $R_J{\otimes}\wedge_\zeta^{\bullet}V_J$ for $S_\zeta V_J{\otimes}\wedge_\zeta^{\bullet}V_J$.
Recall that 
$\wedge_\zeta V_J$ is the quotient of 
$TV_J$ by the two-sided ideal generated by
$\xi_{1r}\otimes\xi_{1r}$ for $r\geqq2$ and 
$\xi_{1r}\xi_{1s}+\zeta^{-1}\xi_{1s}\xi_{1r}$ for 
$2\leqq r<s$.
We can easily show that the $U_{K,\zeta}(\Gp_J)$-bimodule structure of $TV_J$ 
given by
\begin{align*}
u(z_1\otimes\dots \otimes z_n)
=&
\sum_{(u)}u_{(0)}z_1\otimes\dots \otimes u_{(n-1)}z_n,
\\
(z_1\otimes\dots \otimes z_n)u
=&
\sum_{(u)}z_1u_{(0)}\otimes\dots \otimes z_nu_{(n-1)}
\end{align*}
for $u\in U_{K,\zeta}(\Gp_J)$, $z_1,\dots, z_n\in V_J$ 
induces a $U_{K,\zeta}(\Gp_J)$-bimodule structure of 
$\wedge_\zeta V_J$.
We can also easily check that 
$\deru_p:R_J{\otimes}
\wedge_\zeta^{p+1}V_J
\to
R_J
{\otimes}
\wedge_\zeta^{p}V_J$ for $p=0,\dots, n-1$ and 
$\epsilon:R_J\to K$ are homomorphisms of 
$U_{K,\zeta}(\Gp_J)$-bimodules, where the 
$U_{K,\zeta}(\Gp_J)$-bimodule structures of 
$R_J
{\otimes}
\wedge_\zeta^{p}V_J$
and $K$ are given by 
\[
u(d\otimes z)=\sum_{(u)}u_{(0)}d\otimes u_{(1)}z,
\quad
(d\otimes z)u=\sum_{(u)}du_{(0)}\otimes zu_{(1)}
\]
for 
$u\in U_{K,\zeta}(\Gp_J)$, $d\in R_J$, 
$z\in \wedge_\zeta^{p}V_J$, 
and 
$u1=1u=\varepsilon(u)1$ for $u\in	U_{K,\zeta}(\Gp_J)$
respectively.
Hence the Koszul complex 
$R_J{\otimes}\wedge_\zeta^\bullet V_J$
gives a resolution of $K$ as an $R_J$-module as well as a $U_{K,\zeta}(\Gp_J)$-bimodule.
If we regard $R_J{\otimes}\wedge_\zeta^p V_J$ and $K$ as $U_{K,\zeta}(\Gp_J)$-modules with respect to the adjoint action
\[
\ad(u)(z)=\sum_{(u)}u_{(0)}z(S^{-1}u_{(1)})
\qquad
(u\in U_{K,\zeta}(\Gp_J),\;
z\in \wedge_\zeta^p V_J\; \text{ or } \;z\in K),
\]
then 
$R_J{\otimes}\wedge_\zeta^\bullet V_J$
gives a resolution of $K$ as an $R_J$-module as well as a $U_{K,\zeta}(\Gp_J)$-module with respect to the adjoint action.
Since the adjoint action of $U_{K,\zeta}(\Gp_J)$ on 
$R_J{\otimes}\wedge_\zeta^p V_J$ is given by 
\begin{align*}
\ad(u)(d\otimes z)=&
\sum_{(u)}
u_{(0)}(d\otimes z)(S^{-1}u_{(1)})
=
\sum_{(u)}
u_{(0)}d(S^{-1}u_{(3)})\otimes u_{(1)}z(S^{-1}u_{(2)})
\\
=&
\sum_{(u)}
u_{(0)}d(S^{-1}u_{(2)})\otimes 
\ad(u_{(1)})(z)
\end{align*}
for 
$u\in U_{K,\zeta}(\Gp_J)$, $d\in R_J$, 
$z\in \wedge_\zeta^{p}V_J$, 
we write 
$R_J\Breve{\otimes}(\wedge_\zeta^p V_J)_\ad$
for
$R_J{\otimes}\wedge_\zeta^p V_J$
when we regard it as a 
$U_{K,\zeta}(\Gp_J)$-module via the adjoint action
(see Section \ref{subsec:tildetensor}).
We have obtained a resolution $R_J\Breve{\otimes}(\wedge_\zeta^\bullet V_J)_\ad$ of $K$ in 
$\Mod^{R_J}_\inte(U_{K,\zeta}(\Gp_J))$.

Note that $R_J$ is a $U_{K,\zeta}(\Gp_J)$-equivariant $K$-subalgebra of $R$.
Since $R$ is a free $R_J$-module, 
we obtain a resolution 
$R\otimes_{R_J}(R_J\otimes\wedge_\zeta^\bullet V_J)
\cong R\otimes\wedge_\zeta^\bullet V_J$
of 
$R\otimes_{R_J}K\cong \overline{R}$.
Considering the adjoint action we have obtained 
the following result (see \eqref{eq:equiv-ext}).
\begin{proposition}
\label{prop:resolution0}
We have a resolution 
$R\Breve{\otimes}(\wedge_\zeta^\bullet V_J)_\ad$
of $\overline{R}_\ad$ in $\Mod^R_\inte(U_{K,\zeta}(\Gp_J))$.
\end{proposition}
Note that 
$O_{K,\zeta}(P_J)=O_{K,\zeta}(G)\otimes_R\overline{R}$ by Proposition \ref{prop:OPJ}.
By applying $O_{K,\zeta}(G)\otimes_R(\bullet)$ 
to $R\otimes\wedge_\zeta^\bullet V_J$
we obtain a complex 
\[
O_{K,\zeta}(G)\otimes_R(R\otimes\wedge_\zeta^\bullet V_J)
=
O_{K,\zeta}(G)\otimes\wedge_\zeta^\bullet V_J.
\]
\begin{proposition}
\label{prop:resolution}
We have a resolution 
$O_{K,\zeta}(G)\Breve{\otimes}
(\wedge_\zeta^\bullet V_J)_\ad$ 
of $O_{K,\zeta}(P_J)_\ad$
in the abelian category $\Mod^{O_{K,\zeta}(G)}_\inte(U_{K,\zeta}(\Gp_J))$.
\end{proposition}
\begin{proof}
By  Corollary \ref{cor:FF} it is sufficient to show that 
$\hat{O}_{K,\zeta}(G)\otimes\wedge_\zeta^\bullet V_J
$ gives a resolution of $\hat{O}_{K,\zeta}(P_J)$.
Since 
$\hat{O}_{K,\zeta}(G)\otimes\wedge_\zeta^\bullet V_J
$ is obtained by 
by applying $\hat{O}_{K,\zeta}(G)\otimes_R(\bullet)$  to $R\otimes\wedge_\zeta^\bullet V_J$, 
this follows from Lemma \ref{lem:R-free}.
\end{proof}
\subsection{}
Let us compute $R^i\Ind^{\Gg,\Gb}(O_{K,\zeta}(P_J)_\ad)$ using the resolution 
$O_{K,\zeta}(G)\Breve{\otimes}
(\wedge_\zeta^\bullet V_J)_\ad$ of
$O_{K,\zeta}(P_J)_\ad$.
In the rest of this paper we simply write $\Ind$ for $\Ind^{\Gg,\Gb}$.
For $a=0,\dots, n-1$ let
\[
M_a
=
[
O_{K,\zeta}(G)\Breve{\otimes}
(\wedge_\zeta^a V_J)_\ad
\to\cdots\to
O_{K,\zeta}(G)\Breve{\otimes}
(\wedge_\zeta^1 V_J)_\ad
\to
O_{K,\zeta}(G)\Breve{\otimes}
(\wedge_\zeta^0 V_J)_\ad
]
\]
be the truncation of the complex
$O_{K,\zeta}(G)\Breve{\otimes}
(\wedge_\zeta^\bullet V_J)_\ad$.
We have a sequence of morphisms
\[
0=M_{-1}\to
M_0\to M_1\to\cdots \to M_{n-1}=O_{K,\zeta}(G)\Breve{\otimes}
(\wedge_\zeta^\bullet V_J)_\ad
\]
of complexes, and distinguished triangles
\[
M_{a-1}\to M_a\to 
O_{K,\zeta}(G)\Breve{\otimes}
(\wedge_\zeta^a V_J)_\ad
[a]
\xrightarrow{+1}
\]
in the derived category
for $a=0,\dots, n-1$.
Applying $R\Ind$ to it we obtain a distinguished triangle 
\begin{equation}
\label{eq:triangle}
R\Ind(M_{a-1})\to R\Ind(M_a)\to 
R\Ind(O_{K,\zeta}(G)\Breve{\otimes}
(\wedge_\zeta^a V_J))[a]
\xrightarrow{+1}.
\end{equation}
\begin{lemma}
\label{lem:step}
We have
\[
R^p\Ind(O_{K,\zeta}(G)\Breve{\otimes}(\wedge_\zeta^a V_J)_\ad)
\cong
\begin{cases}
O_{K,\zeta}(G)\qquad&(p=a)
\\
0&(p\ne a).
\end{cases}
\]
\end{lemma}
\begin{proof}
By Lemma \ref{lem:TI}  it is sufficient to show
\[
R^p\Ind((\wedge_\zeta^a V_J)_\ad)
\cong
\begin{cases}
K\qquad&(p=a)
\\
0&(p\ne a).
\end{cases}
\]
Note that $(\wedge_\zeta^a V_J)_\ad$ is a sum of one-dimensional weight spaces with weight
\[
-(\alpha_{1j_1}+\cdots+\alpha_{1j_a})
\qquad(2\leqq j_1<j_2<\dots<j_a\leqq n).
\]
Hence the desired result is a consequence of Lemma \ref{lem:step2} below.
\end{proof}
\begin{lemma}
\label{lem:step2}
We have
\[
R^p\Ind(K_{-(\alpha_{1j_1}+\dots+\alpha_{1j_a})})
\cong
\begin{cases}
K\qquad&(p=a, j_1=2,\dots, j_a=a+1)
\\
0&(\text{otherwise})
\end{cases}
\]
for $2\leqq j_1<j_2<\dots<j_a\leqq n$.\end{lemma}
\begin{proof}
By Proposition \ref{prop:BBW0}, 
Proposition \ref{prop:BBW} and  
Lemma \ref{lem:Kos}
we have only to show
\[
\exists
w\in W
\;\;\text{s.t.}\;\; 
\{\alpha_{1j_1},\dots,\alpha_{1j_a}\}
=\Delta^+\cap(-w^{-1}\Delta^+)
\;\Longleftrightarrow\;
j_r=r+1\;\;(r=1,\dots,a).
\]
Assume $\{\alpha_{1j_1},\dots,\alpha_{1j_a}\}
=\Delta^+\cap(-w^{-1}\Delta^+)$ for $w\in W$.
Then we have $w(\alpha_j)\in\Delta^+$ for any $j\in J$.
This implies $\ell(ws_j)>\ell(w)$ for any $j\in J$.
Hence by $W=\GS_n$ and $W_J\cong\GS_{n-1}$ there exists
some $0\leqq t\leqq n-1$ satisfying
$w=s_{t}\dots s_2s_{1}$.
Then we have
$\Delta^+\cap(-w^{-1}\Delta^+)
=\{\alpha_{12},\alpha_{13},\dots, \alpha_{1,t+1}\}$, and hence 
we obtain $t=a$ and 
$j_r=r+1$ for $r=1,\dots,a$.
On the other hand
we have
$
\Delta^+\cap(-w^{-1}\Delta^+)
=\{\alpha_{12},\dots,\alpha_{1,a+1}\}
$ 
for 
$w=s_{a}\cdots s_{1}$.
\end{proof}
By Lemma \ref{lem:step} and \eqref{eq:triangle} we conclude that 
\[
H^p(R\Ind(M_a))=0\qquad(p\ne0).
\]
Moreover, we have an exact sequence
\[
0\to 
H^0(R\Ind(M_{a-1}))
\to
H^0(R\Ind(M_a))
\to
R^a\Ind(O_{K,\zeta}(G)\Breve{\otimes}(\wedge_\zeta^a V_J)_\ad)
\to0.
\]
In particular, we obtain 
\[
R^p\Ind(O_{K,\zeta}(P_J)_\ad)
\cong
H^p(R\Ind(O_{K,\zeta}(G)\Breve{\otimes}(\wedge_\zeta^\bullet V_J)_\ad))
=H^p(R\Ind(M_{n-1}))
=0
\]
for $p\ne0$.
It remains to show that \eqref{eq:oUpsilon} is an isomorphism.

Set
\[
F_a=\Image(H^0(R\Ind(M_a))\to \Ind(O_{K,\zeta}(P_J)_\ad)).
\]
By the above argument we have
a sequence
\[
\{0\}=F_{-1}\subset
F_0\subset
\dots
\subset
F_{n-1}=\Ind(O_{K,\zeta}(P_J)_\ad)
\]
of submodules of $\Ind(O_{K,\zeta}(P_J)_\ad)$ such that $F_a/F_{a-1}\cong R^a\Ind(O_{K,\zeta}(G)\Breve{\otimes}(\wedge_\zeta^a V_J)_\ad)$.
\subsection{}

The following result is well-known and follows from a more general result in \cite{St}.
We present here a direct proof for the sake of the readers.
\begin{lemma}
The $O_{K,\zeta}(H)^{W\circ}$-module
$O_{K,\zeta}(H)^{W_J\circ}$ has a free basis
\[
\{\chi_{a\epsilon_1}\mid a=0,\dots, n-1\}.
\]
\end{lemma}
\begin{proof}
It is sufficient to show that 
$K[\Lambda]^{W_J}$ is a free $K[\Lambda]^W$-module with basis 
$\{e(a\epsilon_1)\mid a=0,\dots, n-1\}$.
Set $\Lambda_0=\Ker(\Lambda\to\Hom_\BZ(Q^\vee,\BZ))$, and 
\[
\varpi_i=\epsilon_1+\dots+\epsilon_i
\quad(i\in I),
\qquad
\varpi'_j=\epsilon_2+\dots+\epsilon_j
\quad(j\in J),
\]
so that 
$\langle\varpi_{i_1},\alpha^\vee_{i_2}\rangle=\delta_{i_1i_2}$ for $i_1, i_2\in I$, and
$\langle\varpi'_{j_1},\alpha^\vee_{j_2}\rangle=\delta_{j_1j_2}$ for $j_1, j_2\in J$.
Then any $\lambda\in\Lambda^+$ is uniquely written in the form
\[
\lambda=
\sum_{i\in I}n_i\varpi_i+\mu
\qquad(n_i\in\BZ_{\geqq0},\; \mu\in\Lambda_0).
\]
Similarly any $\lambda\in\Lambda_J^+=\{
\lambda\in\Lambda\mid
\langle\lambda,\alpha_j^\vee\rangle\geqq0
\;(j\in J)\}$ 
 is uniquely written in the form
\[
\lambda=
m\epsilon_1+
\sum_{j\in J}n_j\varpi'_j+\mu
\qquad(m\in\BZ, n_j\in\BZ_{\geqq0},\; \mu\in\Lambda_0).
\]
Hence we see easily that 
\begin{align*}
K[\Lambda]^W
=&\BZ[\sigma_1,\dots,\sigma_{n-1}]\otimes\BZ[\Lambda_0],
\\
K[\Lambda]^{W_J}
=&
\BZ[e(\epsilon_1)^{\pm1}]
\otimes\BZ[\sigma'_2,\dots,\sigma'_{n-1}]
\otimes\BZ[\Lambda_0],
\end{align*}
where
\[
\sigma_i=\sum_{\lambda\in W(\varpi_i)}e(\lambda)
\quad(i\in I),
\qquad
\sigma'_j=\sum_{\lambda\in W_J(\varpi'_{j})}e(\lambda)
\quad(j\in J).
\]
Therefore our assertion for general $G$ is equivalent to that for $G=SL_n(\BC)$, and is also equivalent to that for $G=GL_n(\BC)$.
We assume that $G=GL_n(\BC)$ in the following.

Set $z_r=e(\epsilon_r)$ for $r=1,\dots, n$ so that $K[\Lambda]=K[z_1^{\pm1},\dots, z_n^{\pm1}]$.
Note that $\sigma_i$ for $i=1,\dots, n-1$ 
(resp.\
$\sigma'_j$ for $j=2,\dots, n-1$)
is the elementary symmetric polynomial of degree $i$
(resp\ $j-1$) in the variables $z_1,\dots, z_n$
(resp.\ $z_2,\dots, z_n$).
Setting $\sigma_n=z_1\dots z_n$ and 
$\sigma'_n=z_2\dots z_n$ we have
\[
K[\Lambda]^W=
K[\sigma_1,\dots, \sigma_{n-1},\sigma_{n}^{\pm1}],
\qquad
K[\Lambda]^{W_J}=
K[z_1^{\pm1}, \sigma'_2,\dots, \sigma'_{n-1},(\sigma'_{n})^{\pm1}].
\]
Let us show that 
$K[\Lambda]^{W_J}$ is a free $K[\Lambda]^{W}$-module with basis 
$\{z_1^a\mid a=0,\dots, n-1\}$.
By 
\begin{gather*}
\sigma'_2=\sigma_1-z_1,\quad
\sigma'_j=\sigma_{j-1}-z_1\sigma'_{j-1}
\;(j\geqq3), 
\quad
\sigma'_n=z_1^{-1}\sigma_n,
\quad
z_1^n+\sum_{k=1}^n(-1)^k\sigma_kz_1^{n-k}=0
\end{gather*}
we see easily that 
$K[\Lambda]^{W_J}=
\sum_{a=0}^{n-1}
K[\Lambda]^{W}z_1^a$.
Hence it is sufficient to show 
\begin{equation}
\label{eq:invariant}
\sum_{a=0}^{n-1}f_az_1^a=0\quad
(f_a\in K[\Lambda]^W)\;
\Rightarrow
f_0=\dots= f_{n-1}=0.
\end{equation}
By multiplying a suitable power of $\sigma_n$ we may assume $f_a\in K[z_1,\dots, z_n]^{W}$. 
Let us show \eqref{eq:invariant} for 
$f_a\in K[z_1,\dots, z_n]^{W}$ by induction on $n$.
Assume 
$\sum_{a=0}^{n-1}f_az_1^a=0$ for 
$f_a\in K[z_1,\dots, z_n]^{W}$. 
By $\sum_{a=0}^{n-1}f_az_2^a=0$ we have 
\[
\sum_{a=1}^{n-1}f_a\frac{z_1^a-z_2^a}{z_1-z_2}=0.
\]
Substituting $z_1=0$ we obtain
$
\sum_{a=1}^{n-1}\overline{f}_az_2^{a-1}=0$,
where $\overline{f}_a=
f_a(0,z_2,\dots, z_n)\in K[z_2,\dots, z_n]^{W_J}$.
By the hypothesis of induction we have $\overline{f}_a=0$.
Namely, $f_a$ is divisible by $z_1$.
Since $f_a$ is $W$-invariant, we have
$f_a=\sigma_ng_a$ for $g_a\in K[z_1,\dots, z_n]^{W}$. 
Then 
$\sum_{a=0}^{n-1}g_az_1^a=0$
Repeating this we conclude that $f_a$ is divisible by any power of $\sigma_n$.
Hence we have $f_a=0$ for any $a$.
We are done.
\end{proof}

For $a=0,\dots, n-1$ we define 
\begin{equation}
\upsilon_a:O_{K,\zeta}(G)_\ad
\otimes K\chi_{a\epsilon_1}
\to
\Ind(O_{K,\zeta}(P_J)_\ad)
\end{equation}
to be the restriction of $\overline{\Upsilon}^{IJ}$ in \eqref{eq:oUpsilon} to 
$O_{K,\zeta}(G)_\ad
\otimes K\chi_{a\epsilon_n}$.
By $\iota_J(\chi_{\epsilon_1})=\xi_{11}$ we have
\[
(\upsilon_a(\varphi\otimes\chi_{a\epsilon_1}))(u)
=
\overline{(\ad(u)(\varphi))\xi_{11}^a}
\qquad(\varphi\in O_{K,\zeta}(G), u\in U_{K,\zeta}(\Gg)).
\]
In order to verify that $\overline{\Upsilon}^{I,J}$ is bijective it is sufficient to show that the image of $\upsilon_a$ is contained in $F_a$ and that the composite of 
\[
O_{K,\zeta}(G)_\ad
\otimes K\chi_{a\epsilon_n}
\xrightarrow{\upsilon_a}
F_a\to F_a/F_{a-1}
\]
is an isomorphism.

Denote by 
\begin{equation}
\sigma_a:
O_{K,\zeta}(G)_\ad
\otimes K\overline{\xi_{11}^a}
\to O_{K,\zeta}(P_J)_\ad
\end{equation}
the restriction of 
\[
O_{K,\zeta}(G)_\ad\otimes\overline{R}_\ad\to
O_{K,\zeta}(P_J)_\ad
\qquad
(\varphi\otimes\overline{f}\mapsto\overline{\varphi f}).
\]
Since $\overline{R}_\ad$ is a 
trivial $U_{K,\zeta}(\Gp_J)$-module (with respect to the adjoint action), 
$\Ind(K\overline{\xi_{11}^a})$ is a one-dimensional trivial $U_{K,\zeta}(\Gg)$-module.
We denote by $1_a$ a $K$-basis of  
$\Ind(K\overline{\xi_{11}^a})$.
Then the composite of
\[
O_{K,\zeta}(G)_\ad
\otimes \Ind(K\overline{\xi_{11}^a})
\cong
\Ind(
O_{K,\zeta}(G)_\ad
\otimes K\overline{\xi_{11}^a})
\xrightarrow{\Ind(\sigma_a)}
\Ind(O_{K,\zeta}(P_J)_\ad)
\]
sends $\varphi\otimes 1_a$ to 
\[
[u\mapsto\overline{((\ad(u)(\varphi))\xi_{11}^a}]\in \Ind(O_{K,\zeta}(P_J)_\ad).
\]
Therefore, in order to verify 
that $\overline{\Upsilon}^{I,J}$ is bijective it is sufficient to show that the image of $\Ind(\sigma_a)$ is contained in $F_a$ and that the composite of 
\[
\Ind(
O_{K,\zeta}(G)_\ad
\otimes K\overline{\xi_{11}^a})
\xrightarrow{\Ind(\sigma_a)}
F_a\to F_a/F_{a-1}
\]
is an isomorphism.

Considering the degrees in the resolution of 
$\overline{R}_\ad$ given in Proposition \ref{prop:resolution0} we obtain a  resolution 
\[
R^0\Breve{\otimes}
(\wedge_\zeta^a V_J)_\ad
\to
R^1\Breve{\otimes}
(\wedge_\zeta^{a-1} V_J)_\ad
\to
\cdots
\to
R^a\Breve{\otimes}
(\wedge_\zeta^{0} V_J)_\ad
\]
of $K\overline{\xi_{11}^a}$.
Tensoring with $O_{K,\zeta}(G)$
we obtain a resolution
\[
N_a=[O_{K,\zeta}(G)
\Breve{\otimes}
(R^0\Breve{\otimes}\wedge_\zeta^a V_J)
\to
\cdots
\to
O_{K,\zeta}(G)\Breve{\otimes}
(R^a\Breve{\otimes}\wedge_\zeta^{0} V_J)]
\]
of $O_{K,\zeta}(G)_\ad\otimes K\overline{\xi_{11}^a}
=O_{K,\zeta}(G)\Breve{\otimes} K\overline{\xi_{11}^a}$.
We have a commutative diagram 
\[
\xymatrix@C=30pt{
N_a
\ar[d]_{\tilde{\sigma}_a}
\ar[r]
&
O_{K,\zeta}(G)_\ad\otimes K\overline{\xi_{11}^a}
\ar[d]^{\sigma_a}
\\
O_{K,\zeta}(G)
\Breve{\otimes}
(\wedge_\zeta^{\bullet} V_J)_\ad
\ar[r]
&
O_{K,\zeta}(P_J)_\ad,
}
\]
where $\tilde{\sigma}_a$ is the morphism of complexes given by
\[
O_{K,\zeta}(G)\Breve{\otimes}
(R^{a-k}\Breve{\otimes}\wedge_\zeta^{k} V_J)
\to
O_{K,\zeta}(G)
\Breve{\otimes}\wedge_\zeta^{k} V_J
\qquad
(\varphi\otimes f\otimes z\mapsto
\varphi f\otimes z).
\]
Since $\tilde{\sigma}_a$ is the composite of 
$N_a\to M_a\to O_{K,\zeta}(G)
\Breve{\otimes}
(\wedge_\zeta^{\bullet} V_J)_\ad$, 
we conclude that 
the image of $\Ind(\sigma_a)$ is contained in $F_a$.
In order to show that 
the composite of 
\[
\Ind(
O_{K,\zeta}(G)_\ad
\otimes K\overline{\xi_{11}^a})
\xrightarrow{\Ind(\sigma_a)}
F_a\to F_a/F_{a-1}
\]
is an isomorphism 
it is sufficient to verify that 
the natural morphism
\[
N_a\to 
O_{K,\zeta}(G)\Breve{\otimes}
(\wedge_\zeta^{a} V_J)_\ad[a]
\]
of complexes induces an isomorphism
\[
R\Ind(N_a)\cong
R\Ind(O_{K,\zeta}(G)\Breve{\otimes}\wedge_\zeta^{a} V_J)[a].
\]
By the definition of $N_a$ we have only  to show
\begin{equation}
\label{eq:F}
R\Ind(O_{K,\zeta}(G)
\Breve{\otimes} (R^{a-k}\Breve{\otimes} 
(\wedge_\zeta^{k} V_J)_\ad))=0
\qquad(k=0,\dots, a-1).
\end{equation}
By Lemma \ref{lem:R} we have
\[
R^{a-k}\Breve{\otimes} 
(\wedge_\zeta^{k} V_J)_\ad
\cong
S^{a-k}_\zeta V\otimes
(\wedge_\zeta^{k} V_J)_\ad
\otimes K_{-(a-k)\epsilon_1},
\]
and hence
\begin{align*}
&R\Ind(O_{K,\zeta}(G)\Breve{\otimes}
(R^{a-k}\Breve{\otimes}
(\wedge_\zeta^{k} V_J)_\ad))
\\
\cong&
O_{K,\zeta}(G)
\Breve{\otimes} 
(S_\zeta^{a-k}V\otimes
R\Ind(
(\wedge_\zeta^{k} V_J)_\ad
\otimes
K_{-(a-k)\epsilon_1}))
\end{align*}
by Lemma \ref{lem:TI} and Lemma \ref{lem:TTI}.
Therefore, \eqref{eq:F} follows from the following.
\begin{lemma}
We have
\[
R\Ind(
(\wedge_\zeta^{k} V_J)_\ad
\otimes
K_{-(a-k)\epsilon_1})=0
\]
for $0\leqq k<a<n$.
\end{lemma}
\begin{proof}
Note that the weights of $
(\wedge_\zeta^{k} V_J)_\ad
\otimes
K_{-(a-k)\epsilon_1}$ is of the form
\[
\lambda=-(a-k)\epsilon_1
-\alpha_{1j_1}-\cdots-\alpha_{1j_k}
\qquad
(2\leqq j_1<j_2<\dots<j_k\leqq n).
\]
Hence it is sufficient to show $R\Ind(K_\lambda)=0$ for $\lambda$ as above.
By 
Proposition \ref{prop:BBW0}  and 
Proposition \ref{prop:BBW}
we have only to show $\langle\lambda+\rho,\alpha^\vee\rangle=0$ for some $\alpha\in\Delta$.
Assume 
$\langle\lambda+\rho,\alpha^\vee\rangle\ne0$ for any $\alpha\in\Delta$.
By 
$\langle\lambda+\rho,\alpha^\vee\rangle\ne0$ for any $\alpha\in\Delta_J$ we have
 $j_s=s+1$ for $s=1,\dots, k$ 
 by the proof of Lemma \ref{lem:step2}.
But then we have 
$\langle\lambda+\rho,\alpha_{1,a+1}^\vee\rangle=0$.
This is a contradiction.
\end{proof}
The proof of Theorem \ref{thm:main} is now complete.

\bibliographystyle{unsrt}

\end{document}